\documentclass{elsarticle}
\usepackage{amsmath,amssymb}
\usepackage{amsthm}
\usepackage{enumerate}

\usepackage{float}
\usepackage{graphicx,here}

\newtheorem{theorem}{Theorem}
\newtheorem{prop}{Proposition}
\newtheorem{cor}{Corollary}
\newdefinition{remark}{Remark}
\newproof{pf}{Proof}

\renewcommand{\hat}{\widehat}
\renewcommand{\tilde}{\widetilde}
\renewcommand{\epsilon}{\varepsilon}

\newcommand{\diag}{\mathop{{\rm diag}}}

\newcommand{\bm}[1]{{\boldsymbol #1}}

\newcommand{\re}{\mathop{\rm Re}}

\newcommand{\cond}{\mathop{\rm cond}\nolimits}

\begin{document}
\begin{frontmatter}
\title{
Mathematical and Numerical Studies
on 
 Meshless Methods
for Exterior Unbounded Domain Problems
}

\author
{Takemi Shigeta}
\ead{shigeta@ntu.edu.tw}
\author
{D.~L.~Young\corref{young}}
\cortext[young]{Corresponding author.}
\ead{dlyoung@ntu.edu.tw}

\address{
Department of Civil Engineering and Hydrotech Research Institute \\
National Taiwan University \\
No.~1, Sec.~4, Roosevelt Road, Taipei 10617, Taiwan}

\begin{abstract}
The method of fundamental solution (MFS) is an efficient meshless method
for solving a boundary value problem in an exterior unbounded domain.
The numerical solution obtained by the MFS is accurate,
while the corresponding matrix equation is ill-conditioned.
A modified MFS (MMFS) with the proper basis functions is proposed
by the introduction of the modified Trefftz method (MTM).
The concrete expressions of the corresponding condition numbers
and 
the solvability by these methods
 are mathematically proven.
Thereby, the optimal parameter minimizing the condition number
is also mathematically given.
Numerical experiments show that the condition numbers of
the matrices corresponding to the MTM and the MMFS are reduced
and that the numerical solution by the MMFS
is more accurate than the one by the conventional method.
\end{abstract}

\end{frontmatter}

\section{Introduction}

The method of fundamental solutions (MFS) is a truly
meshless numerical method for easily and rapidly
solving boundary value problems of elliptic type.
The approximate solution by the MFS is expressed by
a linear combination of fundamental solutions
to a partial differential equation.
Hence, the approximate solution automatically satisfies
the partial differential equation, and
it remains to consider boundary conditions.
In this sense, the MFS is a boundary method
like the boundary element method (BEM).
The BEM requires singular integrals, while
the MFS does not require any treatments for the singularity
of the fundamental solutions,
which is an advantage of the MFS.

Mathon and Johnston \cite{Mathon} first obtained 
numerical solutions by applying the MFS.
Bogomolny \cite{Bogomolny} proved the convergence property.
Katsurada and Okamoto \cite{KO}, \cite{Katsurada} proved
uniqueness of numerical solutions and the convergence property,
and mathematically
discussed the suitable distribution of source points.


Another advantage of the MFS is to directly solve
exterior problems, which is
different from the finite element method (FEM)
and other meshless methods using radial basis functions.
Actually, the MFS can be applied directly to the exterior problems
to obtain accurate solutions.
Even though the basis functions used in the MFS do not satisfy
a condition at infinity,
the accuracy is not bad
in the whole computational domain.
Katsurada~\cite{Katsurada2} proposed 
that the approximate solution for the exterior problems
should be defined by a linear combination of
the proper basis functions satisfying
the governing equation as well as the condition at infinity.
Although many researchers seem to still use the MFS
with the conventional basis functions,
we should use the MFS with the proper basis functions
to obtain more accurate solutions in the whole domain.

On the other hand, the Trefftz method is also
known as a numerical meshless method for solving boundary value problems.
The approximate solution by the Trefftz method is
expressed by a linear combination of functions
satisfying the governing equation.
Hence, similar to the MFS, it is sufficient
if the boundary condition is considered.
This method is also regarded as a boundary method.

It is well known that 
the coefficient matrices derived from the MFS and the Trefftz method
are highly ill-conditioned.
Numerical solutions to the ill-conditioned matrix equations
are unstable.
We need to improve the ill-conditioning.

Ramachandran \cite{Rama} applied the singular
value decomposition (SVD) to the MFS
to improve the accuracy.
Chen {\it et al.}~\cite{CJT1}, \cite{CJT2} discussed 
the equivalence between the Trefftz method and the MFS.
Liu~\cite{Liu} proposed the modified MFS (MMFS)
for interior problems
by introducing the modified Trefftz method (MTM).
Both of the MTM and the MMFS
 can drastically reduce the condition numbers of
the corresponding matrices.


In this paper, we propose a modified MFS
for solving the exterior problem
according to the papers cited above.
The solution of the exterior problem defined
 in Section~\ref{sec:problem}
 is discretized by the MFS
with the proper basis functions satisfying the condition at infinity
in Section~\ref{sec:mfs}.
Section~\ref{sec:mmfs} proposes a modified MFS
based on the modified Trefftz method with
a characteristic length, which plays an important role
for the condition number.
Section~\ref{sec:math} shows some mathematical results
about the condition numbers and the uniqueness of the 
approximate solutions by the MTM and MMFS.
Thereby, the optimal characteristic length that minimizes the condition 
number corresponding to the MTM can explicitly be given
in a mathematical form.
Section~\ref{sec:numerical} shows the effectiveness
of the MTM and the MMFS through numerical experiments.

\section{Problem Setting}\label{sec:problem}
Let $\Omega_0$ be a two dimensional bounded domain enclosed by
the boundary
$\Gamma=\{(r,\theta):r=\rho(\theta),\ 0\leq\theta<2\pi\}$
in the polar representation,
where $\rho(\theta)$ is a smooth function.
Let $\Omega:=\bm R^2\setminus\overline{\Omega_0}$ be
the unbounded domain outside the domain $\Omega_0$.
Then, we consider the following exterior Dirichlet problem
of the Laplace equation:
find $u\in C^2(\Omega)\cap C(\overline\Omega)$ such that
\begin{alignat}{7}
\Delta u&=0 &\quad&\text{in}&\quad&\Omega, \label{eq:lap}\\
u&=f&&\text{on}&&\Gamma,\label{eq:dbc}\\
u(\bm x)&=O(|\bm x|^{-1}) &&\text{as}&& |\bm x|\to\infty,
\label{eq:inf}
\end{alignat}
where $f$ is a continuous function defined on $\Gamma$.

From now on, we identify the Cartesian plane $\bm R^2$
 with the complex plane $\bm C$.
We denote the $(j,k)$ component
of a matrix $Q$ by $Q_{j,k}$.

\section{Discretization by the Method of Fundamental Solutions (MFS)}
\label{sec:mfs}
\subsection{The conventional MFS with the conventional basis functions}
The fundamental solution of the Laplace equation
in two dimensions is defined as
\[ G^*(r):=-\frac 1{2\pi}\ln r \]
for $r=|z|=\sqrt{x^2+y^2}$,
which is a solution to
\[ -\Delta G^*(r)=\delta(r) \]
with the Dirac delta distribution $\delta$.

We distribute
the source points
$\{\zeta_j\}_{j=1}^{N}$
along a circle
outside the domain $\Omega$.
The basis functions are defined as
\begin{equation}
 G_j(z):=-2\pi G^*(|z-\zeta_j|)=\ln|z-\zeta_j|;
\qquad \{\zeta_j\}_{j=1}^{N}\subset\overline{\Omega}^c=\Omega_0.
\label{eq:basis}
\end{equation}
Then, the exact solution $u$ can be approximated
by a linear combination of the basis functions $\{G_j\}_{j=1}^N$
as follows:
\begin{equation}
 u(z)\approx u_{N}(z)
:=\sum_{j=1}^{N}w_jG_j(z),
\qquad \forall z\in\overline\Omega,
 \label{eq:aprx21}
\end{equation}
where $\{w_j\}_{j=1}^N\subset\bm R$ are expansion coefficients
to be determined below.
Since the basis functions (\ref{eq:basis})
have no singular points in $\Omega$,
the approximate solution $u_N$ satisfies
the Laplace equation (\ref{eq:lap}) in the domain $\Omega$.
It remains to consider the boundary condition
to find $\{w_j\}_{j=1}^N$.
We use the boundary collocation method
since it is impossible that $u_N$ exactly 
satisfies the boundary condition.
We distribute
 the collocation points
$\{z_k\}_{k=1}^N$ 
on the boundary $\Gamma$.
Substituting (\ref{eq:aprx21}) 
at the collocation points
into (\ref{eq:dbc}),
we have
\[
\sum_{j=1}^{N}w_jG_j(z_k)=f(z_k),
\quad k=1,2,\ldots, N;\qquad
\{z_k\}_{k=1}^N\subset\Gamma
\]
or in the matrix form:
\begin{equation}
A\bm w=\bm f, \label{eq:acb}
\end{equation}
where the matrix $A=(A_{k,j})\in\bm R^{N\times N}$
and the vectors
$\bm w=(w_j)\in\bm R^{N}$, $\bm f=(f_k)\in\bm R^{N}$
are defined by
\begin{alignat*}{7}
A_{k,j}&:=G_j(z_k), &\quad& k=1,2,\ldots,N;\ j=1,2,\ldots,N, \\
f_k&:=f(z_k), && k=1,2,\ldots, N.
\end{alignat*}
If $A$ is not singular, we can solve (\ref{eq:acb})
and obtain the approximate solution $u_N$
by substituting $\bm w$ into (\ref{eq:aprx21}).

For any point $z=re^{i\theta}\in\Omega$
and the source point $\zeta_j=Re^{i\phi_j}\in\Omega_0$
with the imaginary unit $i=\sqrt{-1}$,
we have
\begin{align*}
|z-\zeta_j|^2
&=|re^{i\theta}-Re^{i\phi_j}|^2
=r^2+R^2-2rR\cos(\theta-\phi_j)\\
&=O(r^2),\qquad r\to\infty,
\end{align*}
from which we know
\[ G_j(z)=\ln|z-\zeta_j|=O(\ln r),\qquad
r\to\infty. \]
Hence, we obtain
\begin{equation}
 u_N(z)=\sum_{j=1}^Nw_jG_j(z)
=O(\ln r),\qquad r\to\infty, \label{eq:cmfs_lnr}
\end{equation}
which does not satisfy the condition (\ref{eq:inf}).
Therefore, the basis functions (\ref{eq:basis}) are
not proper for the exterior problem.

\subsection{The conventional MFS with modified basis functions}
We need to use another basis functions
to approximate a solution that converges to zero at infinity.
We define the following modified basis functions~\cite{Katsurada2}:
\begin{equation}
 \hat G_j(z)=\ln\frac{|z-\zeta_j|}{|z|}
=\ln|z-\zeta_j|-\ln|z|, \label{eq:mbf}
\end{equation}
where $\hat G_j(z)$ satisfies
$\Delta\hat G_j=0$ in $\Omega$
since both of $\ln|z-\zeta_j|$ and $\ln|z|$ are 
the fundamental solutions.
For any point $z=re^{i\theta}\in\Omega$
and the source point $\zeta_j=Re^{i\phi_j}\in\Omega_0$,
we have
\begin{align*}
\frac{|z-\zeta_j|^2}{|z|^2}
&=\frac{r^2+R^2-2rR\cos(\theta-\phi_j)}{r^2}\\
&=1+\left(\frac Rr\right)^2-2\left(\frac Rr\right)
\cos(\theta-\phi_j)\\
&=1+O(r^{-1}),\qquad r\to\infty,
\end{align*}
from which we know
\begin{equation}
\widehat G_j(z)=\ln\frac{|z-\zeta_j|}{|z|}
=\ln\sqrt{1+O(r^{-1})}=O(r^{-1}),\qquad r\to\infty.
\label{eq:exbasis}
\end{equation}
Hence, we obtain
\[
u_N(z)=\sum_{j=1}^N w_j \widehat G_j(z)=
O(r^{-1}),\qquad r\to\infty,
\]
which satisfies the condition (\ref{eq:inf}).
Therefore, we can see that (\ref{eq:mbf}) are
the proper basis functions for solving the exterior problem.
In this paper, the approximate solution
\begin{equation}
\widehat u_N(z)=\sum_{j=1}^N w_j \widehat G_j(z),
\qquad \forall z\in\overline\Omega
\label{eq:cmfs-mbf}
\end{equation}
is called the MFS with the modified basis functions.
When we use this method, in (\ref{eq:acb}) we replace the matrix $A$
into $\hat A$ defined by $\hat A_{k,j}:=\hat G_j(z_k)$.

\section{A modified MFS with modified basis functions}
\label{sec:mmfs}
\subsection{The modified Trefftz method (MTM)}
The approximate solution by
the modified Trefftz method (MTM) \cite{Liu} is given by
\begin{equation}
 \widetilde u_M(r,\theta)
:=a_0+\sum_{k=1}^M\left(
a_k\cos k\theta
+b_k\sin k\theta
\right)\left(\frac {R_0}r\right)^k,\quad
 r\geq \rho(\theta),\quad 0\leq\theta<2\pi,
\label{eq:mtm2_ex}
\end{equation}
where the characteristic length $R_0$ 
is often taken as
\begin{equation}
 R_0\leq \rho_{\rm min}:=\min_{0\leq\theta<2\pi}\rho(\theta).
\label{eq:r0cond}
\end{equation}
This method is reduced to the conventional Trefftz method
if $R_0=1$.
We can see that (\ref{eq:inf}) implies that $a_0=0$.
But, we treat $a_0$ as an unknown coefficient
for reasons of expediency (see Remark~\ref{rem:matrixs}).
Using the boundary collocation method, we have
\begin{equation}
 \widetilde u_M(\rho_j,\theta_j)
=a_0+\sum_{k=1}^M
(a_k\cos k\theta_j
+b_k\sin k\theta_j)\left(\frac {R_0}{\rho_j}\right)^k
=g_j, \quad j=1,2,\ldots,N,
\label{eq:mtm_ex}
\end{equation}
where 
\[
\rho_j:=\rho(\theta_j),\quad
 g_j:=f(\theta_j),\quad
\theta_j:=2\pi (j-1)/N.
\]
Therefore, we obtain the matrix equation:
\begin{equation}
 S\bm y=\bm g,
\label{eq:syb}
\end{equation}
where
\begin{gather*}
 \bm y:=(a_0,a_1,b_1,\ldots,a_M,b_M)^T\in\bm R^{2M+1},\qquad
\bm g:=(g_1,g_2,\ldots,g_N)^T\in\bm R^N
\end{gather*}
and $S\in\bm R^{N\times(2M+1)}$ is defined by
\begin{multline*}
S_{j,1}:=1,\quad
S_{j,2k}:=\left(\frac{R_0}{\rho_j}\right)^k\cos k\theta_j,\quad
S_{j,2k+1}:=\left(\frac{R_0}{\rho_j}\right)^k\sin k\theta_j \\
(j=1,2,\ldots,N;\ k=1,2,\ldots,M).
\end{multline*}
When $R_0$ satisfies (\ref{eq:r0cond}), 
every element of $S$
is less than or equal to 1.
We take $N=2M+1$ to make $S$ a square matrix.

Let $z_j=\rho_j e^{i\theta_j}$ in (\ref{eq:acb}).
Then, $\bm f$ coincides with $\bm g$.

\begin{remark}\label{rem:matrixs}
Since we know that $a_0=0$, the number of unknown coefficients
are $2M$ in essentials.
We can remove the first column
of $S$, which we denote by $S'$.
If we consider the square matrix $S'\in\bm R^{2M\times N}$
with $N=2M$,
then we can see that
\[
\sin M\theta_j=\sin M\frac{2\pi(j-1)}{N}=\sin \pi(j-1)=0
\qquad (j=1,2,\ldots,N),
\]
which implies $S'_{j,2M}=0$ for $j=1,2,\ldots,N$.
Hence, $S'$ becomes singular.
Although we can avoid the singularity of $S'$
by taking $\theta_j=2\pi(j-1)/(N+1)$,
we treat $a_0$ as an unknown coefficient 
to use $S$ in this paper.
\end{remark}

\subsubsection{Transformation of the MFS to the MTM}
The modified basis function for the source point
$\zeta=Re^{i\phi}$
can be transformed as follows:
\begin{align*}
\ln\frac{|z-\zeta|}{|z|}
&=\ln\frac{|re^{i\theta}-Re^{i\phi}|}{|re^{i\theta}|}
=\ln\left|1-\frac Rre^{i(\phi-\theta)}\right|\\
&=\re\ln\left(1-\frac Rre^{i(\phi-\theta)}\right)
=-\re\sum_{k=1}^\infty\frac 1k
\left(\frac Rre^{i(\phi-\theta)}\right)^k\\
&=-\re\sum_{k=1}^\infty\frac 1k
\left(\frac Rr\right)^ke^{ik(\phi-\theta)}
=-\sum_{k=1}^\infty\frac 1k
\left(\frac Rr\right)^k\cos k(\phi-\theta)\\
&=-\sum_{k=1}^\infty\frac 1k
\left(\frac Rr\right)^k\cos k(\theta-\phi).
\end{align*}
Hence, we can write the approximate solution by the MFS in the form:
\begin{align}
u_N(r,\theta)&=\sum_{j=1}^Nw_j
\ln\frac{|z-\zeta_j|}{|z|} \label{eq:mfs_c_ex} \\
&=\sum_{j=1}^Nw_j\left[
-\sum_{k=1}^\infty\frac 1k
\left(\frac Rr\right)^k\cos k(\theta-\theta_j)
\right], \label{eq:mfs_tm_ex}
\end{align}
where we put $\zeta_j=Re^{i\theta_j}$ with 
$\theta_j=2\pi(j-1)/N$ ($j=1,2,\ldots,N$).
If we use the conventional basis functions,
the approximate solution can be written in the form:
\begin{align}
u_N(r,\theta)&=\sum_{j=1}^Nw_j
\ln|z-\zeta_j| \label{eq:mfs_c_ex2} \\
&=\sum_{j=1}^Nw_j\left[
\ln r-\sum_{k=1}^\infty\frac 1k
\left(\frac Rr\right)^k\cos k(\theta-\theta_j)
\right]. \label{eq:mfs_tm_ex2}
\end{align}
We know from (\ref{eq:cmfs_lnr})
that (\ref{eq:mfs_c_ex2}) or (\ref{eq:mfs_tm_ex2})
does not satisfy (\ref{eq:inf}).
But, if we impose the condition
\[
 \sum_{j=1}^Nw_j\ln r=0 \qquad (r\neq 1)
\]
or
\begin{equation}
 \sum_{j=1}^Nw_j=0,
\label{eq:sumw}
\end{equation}
then (\ref{eq:mfs_tm_ex2}) coincides with (\ref{eq:mfs_tm_ex}).
Hence, we know that (\ref{eq:mfs_tm_ex2})
satisfies (\ref{eq:inf}) as long as (\ref{eq:sumw}) is imposed.

On the boundary
\[
\Gamma=\{(r,\theta):r=\rho(\theta),\quad 0\leq\theta<2\pi\},
\]
the approximate solution (\ref{eq:mfs_tm_ex}) is written as
\begin{align}
u_N(\rho(\theta),\theta)
&=\sum_{j=1}^Nw_j\left[
-\sum_{k=1}^\infty
\frac 1k
\left(\frac R{\rho(\theta)}\right)^k\cos k(\theta-\theta_j)
\right] \notag\\
&=
-\sum_{k=1}^\infty\sum_{j=1}^Nw_j
\frac 1k
\left(\frac R{\rho(\theta)}\right)^k
(\cos k\theta\cos k\theta_j
+\sin k\theta\sin k\theta_j) \notag\\
&=
-\sum_{k=1}^\infty\sum_{j=1}^Nw_j
\frac 1k
\left(\frac R{R_0}\right)^k
(\cos k\theta_j\cos k\theta
+\sin k\theta_j\sin k\theta)
\left(\frac{R_0}{\rho(\theta)}\right)^k.
\label{eq:mmfs2_ex}
\end{align}
Truncating the infinite series $\sum_{k=1}^\infty$
in (\ref{eq:mmfs2_ex})
into a finite series $\sum_{k=1}^M$:
\begin{equation}
u_N(\rho(\theta),\theta)
\approx
-\sum_{k=1}^M\sum_{j=1}^Nw_j
\frac 1k
\left(\frac R{R_0}\right)^k
(\cos k\theta_j\cos k\theta
+\sin k\theta_j\sin k\theta)
\left(\frac{R_0}{\rho(\theta)}\right)^k, \label{eq:mmfs2m_ex}
\end{equation}
and comparing (\ref{eq:mmfs2m_ex})
 with (\ref{eq:mtm2_ex}), we obtain
\begin{align}
a_k&=-\frac 1k\left(\frac{R}{R_0}\right)^k\sum_{j=1}^Nw_j
\cos k\theta_j,\qquad k=1,2,\ldots,M,\label{eq:ak} \\
b_k&=-\frac 1k\left(\frac{R}{R_0}\right)^k\sum_{j=1}^Nw_j
\sin k\theta_j,\qquad k=1,2,\ldots,M.\label{eq:bk}
\end{align}
Adding the condition (\ref{eq:sumw}) to (\ref{eq:ak}) and (\ref{eq:bk}),
we obtain the following matrix equation:
\begin{equation}
 \bm y=K\bm w, \label{eq:yKw}
\end{equation}
where $K\in \bm R^{(2M+1)\times N}$ is defined by
\begin{multline*}
K_{1,j}:=1,\quad
K_{2k,j}:=-\frac 1k\left(\frac R{R_0}\right)^k\cos k\theta_j,\quad
K_{2k+1,j}:=-\frac 1k\left(\frac R{R_0}\right)^k\sin k\theta_j \\
(k=1,2,\ldots,M;\ j=1,2,\ldots,N).
\end{multline*}
Substituting (\ref{eq:yKw}) into (\ref{eq:syb}),
we obtain the matrix equation
\begin{equation}
SK\bm w=\bm f.
\label{eq:skwb}
\end{equation}

\begin{remark}
When we consider (\ref{eq:mfs_tm_ex}),
it is not necessary to impose (\ref{eq:sumw})
in essentials.
In addition, we know that $a_0=0$.
Hence, we can remove the first row of $K$, which is denoted by
$K'\in\bm R^{2M\times N}$.
Then, the square matrix $K'$ with $N=2M$
becomes singular
because of the same reason described in Remark~\ref{rem:matrixs}.
Therefore, we impose the condition (\ref{eq:sumw})
to use $K$.
\end{remark}

The matrix $K$ for $R_0=1$, denoted by $K_1$, becomes
\begin{multline*}
K_{1,j}=1,\quad
K_{2k,j}=-\frac 1kR^k\cos k\theta_j,\quad
K_{2k+1,j}=-\frac 1kR^k\sin k\theta_j \\
(k=1,2,\ldots,M;\ j=1,2,\ldots,N),
\end{multline*}
while $K$ for $R_0=R$, denoted by $K_2$, becomes
\begin{multline*}
K_{1,j}=1,\quad
K_{2k,j}=-\frac 1k\cos k\theta_j,\quad
K_{2k+1,j}=-\frac 1k\sin k\theta_j \\
(k=1,2,\ldots,M;\ j=1,2,\ldots,N),
\end{multline*}
which is independent of $R$.

After obtaining the unknown vector $\bm w$,
we substitute $\bm w$ into (\ref{eq:mfs_c_ex})
not (\ref{eq:mmfs2m_ex}).
We call this method the modified MFS (MMFS).
Even if we substitute $\bm w$ into (\ref{eq:mmfs2m_ex}),
the result should be the same as the result when we use
(\ref{eq:mtm_ex}).

\begin{remark}
We can see that $SK\in \bm R^{N\times N}$ is always a square matrix
for any $M$.
Since we would like to find unknown $\bm w$ from 
the known $\bm f$,
 it is not necessary that $S$ and $K$
 are square matrices.
However, we see from numerical experiments that
a better approximate solution can be obtained
when the matrices $S$ and $K$ are square.
This numerical result is not shown in this paper.

We notice that the matrix $SK$ converges to 
the matrix $A$ corresponding to
the conventional MFS as $M\to\infty$.
In this sense, $SK$ can be regarded as an approximation
to $A$.
\end{remark}

\section{Some Mathematical Results}\label{sec:math}
In this section, 
we show some theoretical results about the matrices $S$ and $K$
defined in the previous section.
In particular, we obtain the optimal $R_0$ 
in a mathematical form.
We prove the uniqueness of
the numerical solutions obtained by the MTM and the MMFS.

The following Propositions \ref{prop:kdec} and \ref{lem:tt}
have already been shown in \cite{CJT1} and \cite{Liu}.
In this paper, we give the complete and more simple proof
of Proposition \ref{lem:tt}.

\begin{prop}\label{prop:kdec}
The matrix $K$ can be decomposed into
\[
 K=T_R T_\theta,
\]
where $T_R\in\bm R^{(2M+1)\times(2M+1)}$ 
and $T_\theta\in\bm R^{(2M+1)\times N}$ are defined as
\begin{align*}
 (T_R)_{1,1}&=1,\qquad
(T_R)_{k,j}=0 \quad (k\neq j)\\
 (T_R)_{2k,2k}&=(T_R)_{2k+1,2k+1}=-\frac 1k\left(\frac R{R_0}\right)^k
\quad (k=1,2,\ldots,M)
\end{align*}
and
\begin{align*}
 (T_\theta)_{1,j}&=1\quad (j=1,2,\ldots,N),\\
 (T_\theta)_{2k,j}&=\cos k\theta_j,\quad
 (T_\theta)_{2k+1,j}=\sin k\theta_j
\quad (k=1,2,\ldots,M;\ j=1,2,\ldots,N).
\end{align*}
\end{prop}

\begin{prop}\label{lem:tt}
The matrix $T_\theta$ has the orthogonal property:
\begin{equation}
 T_\theta T_\theta^T=\frac N2\diag(2,1,\ldots,1)
\in\bm R^{(2M+1)\times(2M+1)}. \label{eq:ttorth}
\end{equation}
When $N=2M+1$, the inverse matrix $T_\theta^{-1}$ can be written as
\begin{align*}
 (T_\theta^{-1})_{j,1}&=\frac 1N\quad (j=1,2,\ldots,N),\\
 (T_\theta^{-1})_{j,2k}&=\frac 2N\cos k\theta_j,\quad
 (T_\theta^{-1})_{j,2k+1}=\frac 2N\sin k\theta_j
\quad (k=1,2,\ldots,M;\ j=1,2,\ldots,N).
\end{align*}
\end{prop}
\begin{proof}
Let $\omega:=\exp(2\pi i/N)$.
Then, we see $\exp(ik\theta_l)=\omega^{k(l-1)}$.
We notice that
\[
 \sum_{l=1}^N\omega^{s(l-1)}=
\left\{
\begin{array}{cc}
N & (s=mN;\ m=0,\pm 1,\pm 2,\ldots)\\
0 & (s\neq mN;\ m=0,\pm 1,\pm 2,\ldots)
\end{array}
\right.
\]
since for $s\neq mN$,
we can divide the following equation by $\omega^s-1\neq 0$:
\[
(\omega^s-1)
\left((\omega^s)^{N-1}+(\omega^s)^{N-2}+\cdots+\omega^s+1
\right)
=(\omega^s)^N-1
=0.
\]
Since $T_\theta T_\theta^T$ is symmetric,
it is sufficient to 
calculate the following components:
\begin{align*}
 (T_\theta T_\theta^T)_{2k,j}
&=\sum_{l=1}^N\cos j\theta_l\cos k\theta_l \\
&=\sum_{l=1}^N
\frac{\omega^{j(l-1)}+\omega^{-j(l-1)}}2
\frac{\omega^{k(l-1)}+\omega^{-k(l-1)}}2 \\
&=\frac 14\sum_{l=1}^N[
\omega^{(j+k)(l-1)}+\omega^{-(j+k)(l-1)}
+\omega^{(j-k)(l-1)}+\omega^{-(j-k)(l-1)}] \\
&=\left\{
\begin{array}{cc}
N/2 & (j=k) \\
0 & (j\neq k)
\end{array}
\right..
\end{align*}
In a similar fashion, we have
\[
 (T_\theta T_\theta^T)_{2k,j+1}=0,\qquad
 (T_\theta T_\theta^T)_{2k+1,j+1}
=\left\{
\begin{array}{cc}
N/2 & (j=k) \\
0 & (j\neq k)
\end{array}
\right..
\]
Hence, we obtain (\ref{eq:ttorth}).

Next, we calculate the inverse of $T_\theta$.
Let the matrix $D$ satisfy
\begin{equation}
 T_\theta T_\theta^T=\frac N2\diag\left(2,1,\ldots,1\right)
=D^2, \label{eq:ttd2}
\end{equation}
from which we can write
\[
 D
=\sqrt\frac N2\diag(\sqrt 2,1,\ldots,1).
\]
Premultiplying and postmultiplying (\ref{eq:ttd2}) by $D^{-1}$
respectively, we have
\[
 D^{-1}T_\theta T_\theta^TD^{-1}=I.
\]
We then know that $D^{-1}T_\theta$ is an orthogonal matrix:
\[
 D^{-1}T_\theta(D^{-1}T_\theta)^T=I,
\]
from which we have $(D^{-1}T_\theta)^{-1}=(D^{-1}T_\theta)^T$.
Therefore, we obtain
\begin{equation}
T_\theta^{-1}=T_\theta^TD^{-2} \label{eq:ttinv}
\end{equation}
with
\[
 D^{-2}:=(D^{-1})^2
=\frac 2N\diag(\frac 12,1,\ldots,1).
\]
We can obtain all the components of $T_\theta^{-1}$ explicitly
from (\ref{eq:ttinv}).
\end{proof}

\begin{remark}
Arbitrary matrix $Q$ can always be decomposed into
\[
 Q=UP,
\]
where $U$ denotes an orthogonal matrix and
$P$ a positive semidefinite symmetric matrix.
This is said to be the polar decomposition of $Q$.

We can redefine the sign of a part of the components
of $T_R$ and $T_\theta$ as follows:
\begin{align*}
 (T_R)_{2k,2k}=(T_R)_{2k+1,2k+1}
=\frac 1k\left(\frac R{R_0}\right)^k\\
(T_\theta)_{2k,j}=-\cos k\theta_j,\quad
(T_\theta)_{2k+1,j}=-\sin k\theta_j
\end{align*}
 in Proposition~\ref{prop:kdec}.
We see that $K$ can be written as $K=(T_RD)(D^{-1}T_\theta)$
with the orthogonal matrix $D^{-1}T_\theta$ 
and the diagonal matrix $T_RD$ whose diagonal components
are positive.
Hence, we know that $K^T=(D^{-1}T_\theta)^T(T_RD)$
 is the polar decomposition.
\end{remark}

In the following proposition,
the determinant $\det(T_\theta)$ with $N=2M+1$
is found in \cite{CJT1} and \cite{Liu}.
\begin{prop}\label{prop:detk}
Let $N=2M+1$.
Then, the determinants of the matrices $T_R$ and $T_\theta$
can be given by
\begin{align*}
 \det(T_R)=\frac 1{(M!)^2}\left(\frac R{R_0}\right)^{M(M+1)},\qquad
 \det(T_\theta)
=\frac{N^{M+1/2}}{2^M}.
\end{align*}
Hence, $K=T_R T_\theta$ is not singular and 
$K^{-1}$ can explicitly be written as
\begin{align*}
 (K^{-1})_{j,1}&=\frac 1N \quad (j=1,2,\ldots,N),\\
 (K^{-1})_{j,2k}&=-\frac{2k}N\left(\frac{R_0}R\right)^k
\cos k\theta_j,\\
 (K^{-1})_{j,2k+1}&=-\frac{2k}N\left(\frac{R_0}R\right)^k
\sin k\theta_j \quad
 (j=1,2,\ldots,N;\ k=1,2,\ldots,M)
\end{align*}
\end{prop}

\begin{proof}
It is easy to obtain $\det(T_R)$ and $\det(T_\theta)$
by simple calculations:
\begin{align*}
 \det(T_R)&=1\cdot\prod_{k=1}^M\left[
-\frac 1k\left(\frac R{R_0}\right)^k\right]^2
=\frac 1{(M!)^2}\left(\frac R{R_0}\right)^{2\sum_{k=1}^Mk} \\
&=\frac 1{(M!)^2}\left(\frac R{R_0}\right)^{M(M+1)}
\end{align*}
and
\[
 \det(T_\theta)
=\sqrt{\det(T_\theta^2)}
=\sqrt{\det(T_\theta T_\theta^T)}
=\left\{N\cdot\left(\frac N2\right)^{2M}\right\}^{1/2}
=\frac{N^{M+1/2}}{2^M}.
\]
From Proposition~\ref{lem:tt}, we can immediately derive
all the components of
$K^{-1}=T_\theta^{-1}T_R^{-1}$.
\end{proof}
We notice that $K$ is independent of the boundary $\Gamma$
and that $K^{-1}$ is explicitly given, we can always
directly obtain $\bm w$
from $\bm y$ without solving (\ref{eq:yKw}).

The condition number of a matrix $Q$ corresponding to the $p$-norm
is defined as
\[
 \cond_p(Q)=\|Q\|_p\|Q^{-1}\|_p,
\]
where $\|\cdot\|_p$ is the matrix norm.
When $p=2$, the matrix norm $\|Q\|_2$ is
the largest singular value of $Q$ or
the square root of the largest eigenvalue of 
$Q^TQ$.

\begin{prop}\label{prop:condk}
Let $N=2M+1$ and $R_0=R$.
Then, the condition numbers of $T_R$, $T_\theta$ and $K=K_2$
corresponding to
 the 2-norm
are given as
\[
  \cond_2(T_R)=M,\qquad
\cond_2(T_\theta)=\sqrt 2,\qquad
 \cond_2(K_2)=\sqrt 2M,
\]
which are independent of $R$.
\end{prop}
\begin{proof}
When we take $R_0=R$, we have
\[
 T_R=\diag\left(1,-1,-1,-\frac 12,-\frac 12,\ldots,
-\frac 1M,-\frac 1M\right).
\]
We can then easily see that $\|T_R\|_2=1$ and $\|T_R^{-1}\|_2=M$,
which follows $\cond_2(T_R)=M$.

We notice that the matrix norm $\|T_\theta\|_2$ is
the square root of the largest eigenvalue of $T_\theta T_\theta^T$.
Hence, from (\ref{eq:ttorth}) in Proposition~\ref{lem:tt},
 we obtain $\|T_\theta\|_2=\sqrt N$.
Similarly, we have $\|T_\theta^{-1}\|_2=\sqrt{2/N}$.
Therefore, we obtain $\cond_2(T_\theta)=\sqrt 2$.

Next, since $T_R$ and $T_\theta T_\theta^T$ are diagonal matrices,
we can calculate
\begin{align*}
 K_2K_2^T&=(T_RT_\theta)(T_RT_\theta)^T=T_R(T_\theta T_\theta^T)T_R\\
&=\frac N2\diag\left(2,1,1,\frac 1{2^2},\frac 1{2^2},\ldots,
\frac 1{M^2},\frac 1{M^2}\right).
\end{align*}
Hence, we can see that $\|K_2\|_2=\sqrt N$ and
 $\|K_2^{-1}\|_2=\sqrt 2M/\sqrt N$,
which follows $\cond_2(K)=\sqrt 2M$.
\end{proof}

We assume that the boundary is a circle.
Then, in a manner similar to the decomposition of $K$,
the matrix $S$ can be decomposed as follows:
\begin{prop}\label{lem:sdec}
Let $\rho(\theta)$ be a constant $\rho$.
Then, the matrix $S$ can be decomposed into
\[
 S
=T_\theta^T S_{R_0},
\]
where $S_{R_0}\in\bm R^{(2M+1)\times(2M+1)}$ 
is defined as
\begin{align*}
 (S_{R_0})_{1,1}&=1,\qquad
(S_{R_0})_{k,j}=0 \quad (k\neq j)\\
 (S_{R_0})_{2k,2k}
&=(S_{R_0})_{2k+1,2k+1}=\left(\frac {R_0}\rho\right)^k
\quad (k=1,2,\ldots,M).
\end{align*}
\end{prop}

We obtain the following proposition and theorem
by similar discussions to Propositions~\ref{prop:detk}
 and \ref{prop:condk}:

\begin{prop}\label{prop:dets}
Let $N=2M+1$.
Under the assumption in Proposition~\ref{lem:sdec},
the determinants of the matrices $S_{R_0}$ and $S$
can be given by
\begin{align*}
 \det(S_{R_0})=\left(\frac {R_0}\rho\right)^{M(M+1)},\qquad
 \det(S)=\frac{N^{M+1/2}}{2^M}
\left(\frac {R_0}\rho\right)^{M(M+1)}.
\end{align*}
Hence, $S$ is not singular.
Namely, the approximate solution by the MTM
is uniquely determined.
\end{prop}

\begin{remark}
We can solve (\ref{eq:syb})
by Proposition \ref{prop:dets}.
Using Propositions \ref{lem:tt} and \ref{lem:sdec},
$\bm y=S_{R_0}^{-1}(T_\theta^T)^{-1}\bm f$ can be written
in the form:
\begin{align*}
 a_0&=\frac 1N\sum_{j=1}^Ng_j,\\
a_k&=\frac 2N\left(\frac\rho{R_0}\right)^k\sum_{j=1}^Ng_j\cos k\theta_j,\\
b_k&=\frac 2N\left(\frac\rho{R_0}\right)^k\sum_{j=1}^Ng_j\sin k\theta_j \quad (k=1,2,\ldots,M),
\end{align*}
which correspond to the approximation to the coefficients of
 the Fourier series (\ref{eq:mtm2_ex}):
\begin{align*}
a_0&=\frac 1{2\pi}\int_0^{2\pi}f(\theta)\,d\theta,\\
a_k&=\frac 1{\pi}\left(\frac\rho{R_0}\right)^k
\int_0^{2\pi}f(\theta)\cos k\theta\,d\theta,\\
b_k&=\frac 1{\pi}\left(\frac\rho{R_0}\right)^k
\int_0^{2\pi}f(\theta)\sin k\theta\,d\theta,\quad
(k=1,2,\ldots,M).
\end{align*}
\end{remark}

\begin{theorem}\label{prop:conds}
Under the assumption in Proposition~\ref{prop:dets},
the condition number of $S$ corresponding to
 the 2-norm is given as
\[
 \cond_2(S)=\left\{
\begin{array}{cl}
\displaystyle \sqrt 2\left(\frac\rho{R_0}\right)^M
& (R_0<\rho) \\
\displaystyle \sqrt 2\left(\frac\rho{R_0}\right)
& (\rho\leq R_0< 2^{1/(2M)}\rho) \\
\displaystyle \left(\frac{R_0}\rho\right)^{M-1}
& (2^{1/(2M)}\rho\leq R_0<\sqrt 2\rho) \\
\displaystyle \frac 1{\sqrt 2}\left(\frac{R_0}\rho\right)^{M}
& (R_0\geq\sqrt 2\rho)
\end{array}
\right..
\]
\end{theorem}
\begin{proof}
Since $S_{R_0}$ and $T_\theta T_\theta^T$ are diagonal matrices,
we can calculate
\begin{align*}
 S^TS&=(T_\theta S_{R_0})^T(T_\theta S_{R_0})
=S_{R_0}(T_\theta T_\theta^T)S_{R_0}\\
&=\frac N2\diag\left(2,
\left(\frac{R_0}\rho\right)^2,\left(\frac{R_0}\rho\right)^2,\ldots,
\left(\frac{R_0}\rho\right)^{2M},\left(\frac{R_0}\rho\right)^{2M}\right).
\end{align*}
Then, we know
\[
 \|S\|_2^2=\frac N2\max\left\{2,
\left(\frac{R_0}\rho\right)^2,\left(\frac{R_0}\rho\right)^{2M}
\right\}
\]
and
\[
 \|S^{-1}\|_2^2=\left[\frac N2\min\left\{2,
\left(\frac{R_0}\rho\right)^2,\left(\frac{R_0}\rho\right)^{2M}
\right\}\right]^{-1}.
\]
Hence, we derive
\[
 \|S\|_2=\left\{
\begin{array}{cl}
\sqrt N & \displaystyle( {R_0}<2^{1/(2M)}\rho ) \\
\displaystyle\sqrt{\frac N2}\left(\frac{R_0}\rho\right)^M &
 \displaystyle( {R_0}\geq 2^{1/(2M)}\rho) 
\end{array}
\right.
\]
and
\[
 \|S^{-1}\|_2=\left\{
\begin{array}{cl}
\displaystyle\sqrt{\frac 2N}\left(\frac\rho{R_0}\right)^M &
 \displaystyle( {R_0}<\rho ) \\
\displaystyle\sqrt{\frac 2N}\left(\frac\rho{R_0}\right) &
 \displaystyle( \rho\leq{R_0}< \sqrt 2\rho ) \\
\displaystyle\frac 1{\sqrt N} &
 \displaystyle( {R_0}\geq \sqrt 2\rho ) 
\end{array}
\right.,
\]
from which we obtain $\cond_2(S)=\|S\|_2\|S^{-1}\|_2$.
\end{proof}

We denote the characteristic length $R_0$
 that minimizes $\cond_2(S)$ by $R_0^{\rm opt}$,
which is called the optimal characteristic length.
From Theorem~\ref{prop:conds},
we obtain $R_0^{\rm opt}$ as follows:
\begin{cor}\label{cor:optr}
Under the assumption in Theorem \ref{prop:conds},
the optimal characteristic length is given by
\[ 
 R_0^{\rm opt}=2^{1/(2M)}\rho.
\]
Then, the minimal condition number of $S$ is written as
\begin{equation}
 \cond_2(S)=2^{(M-1)/(2M)}. \label{eq:mincond}
\end{equation}
\end{cor}
We can see that 
(\ref{eq:mincond}) is a monotonically increasing function of $M$
and is less than $\sqrt 2$.
Hence, the minimal condition number is less than
 $\cond_2(S)=\sqrt 2$ for $R_0=\rho$.
On the other hand, if we consider the conventional Trefftz method
($R_0=1$) for the circular boundary whose radius is greater than 1
($\rho>1$), from Theorem \ref{prop:conds}
we have $\cond_2(S)=\sqrt 2\rho^M$,
which exponentially diverges for large $M$ and $\rho$.
Therefore, we know that the MTM is efficient.

From Propositions~\ref{prop:kdec} and \ref{lem:sdec},
we immediately obtain the following theorem:
\begin{theorem}
Under the assumption in Theorem~\ref{prop:conds},
the matrix $SK$ corresponding to the MMFS
can be decomposed into
\[
 SK=T_\theta^T \Lambda T_\theta
\]
with the diagonal matrix 
$\Lambda:=S_{R_0}T_R\in\bm R^{(2M+1)\times(2M+1)}$ defined as
\begin{align*}
 \Lambda_{1,1}&=1,\qquad
\Lambda_{k,j}=0 \quad (k\neq j)\\
 \Lambda_{2k,2k}
&=\Lambda_{2k+1,2k+1}=-\frac 1k\left(\frac R\rho\right)^k
\quad (k=1,2,\ldots,M).
\end{align*}
Therefore, the approximate solution by the MMFS
is uniquely determined.
\end{theorem}

\section{Numerical Experiments}\label{sec:numerical}
\subsection{Condition number corresponding to the MTM}
We consider the following three domains (Figure~\ref{fig:domain}):
\[
\Omega_k=\{(r,\theta):r=\rho^{(k)}(\theta),\ 0\leq\theta<2\pi\},
\qquad k=1,2,3,
\]
where $\rho^{(k)}(\theta)$ ($0\leq\theta<2\pi$)
are defined as
\begin{enumerate}[(a)]
\item circle: $\rho^{(1)}(\theta)=1$;
 $(\rho_{\rm min},\rho_{\rm max})=(1,1)$
\item ellipse:
$\rho^{(2)}(\theta)=ab/\sqrt{a^2\sin^2\theta+b^2\cos^2\theta}$
with $a=10,b=5$;
 $(\rho_{\rm min},\rho_{\rm max})=(5,10)$
\item epitrochoid:
$\rho^{(3)}(\theta)=\sqrt{(a+b)^2+1-2(a+b)\cos(a\theta/b)}$
with $a=3,b=1$;
 $(\rho_{\rm min},\rho_{\rm max})=(3,5)$
\end{enumerate}
with $\displaystyle\rho_{\rm min}=\min_{0\leq\theta<2\pi}\rho(\theta)$
and
$\displaystyle\rho_{\rm max}=\max_{0\leq\theta<2\pi}\rho(\theta)$.
\begin{figure}[H]
\centering
\includegraphics[width=0.3\textwidth]{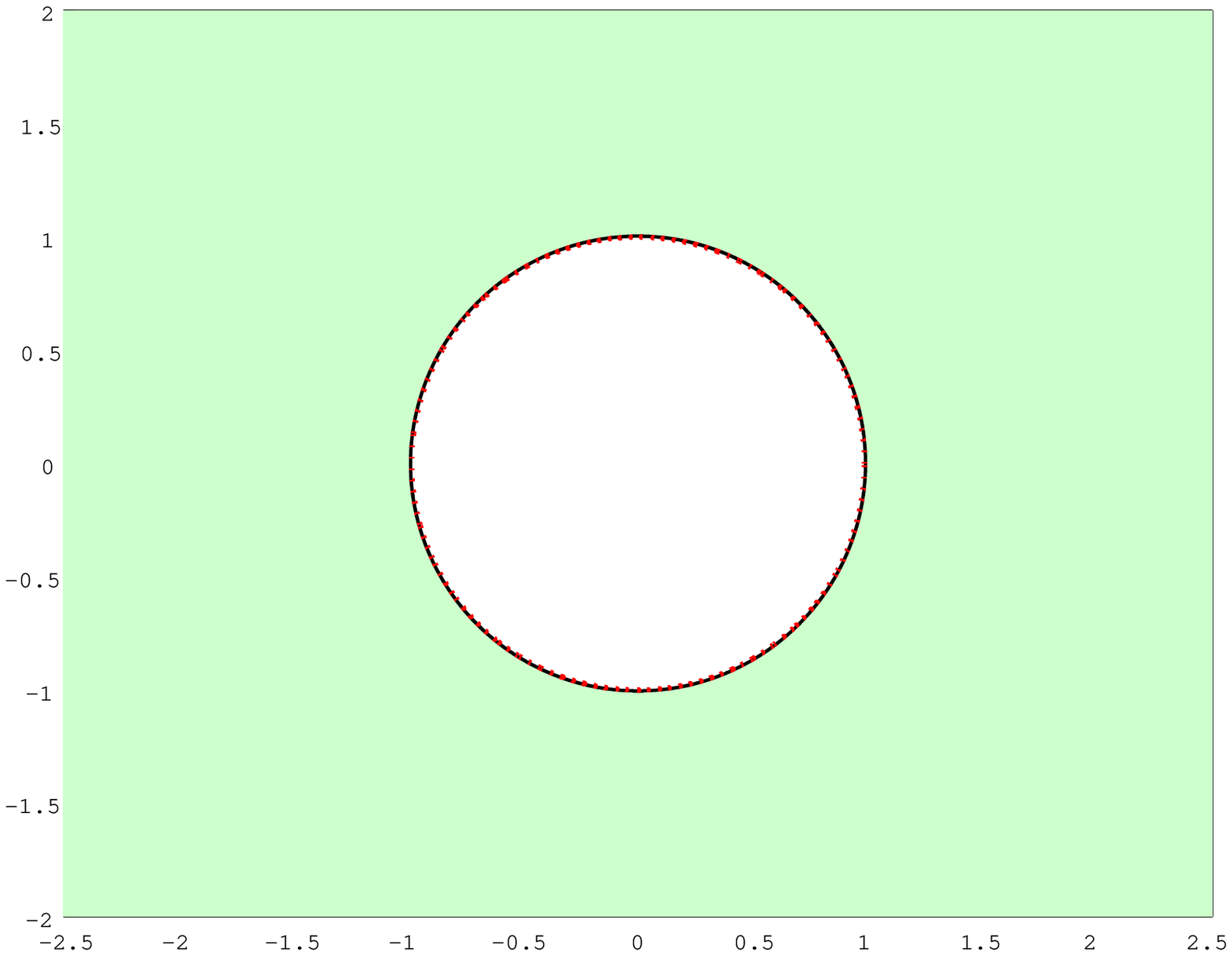}
\includegraphics[width=0.3\textwidth]{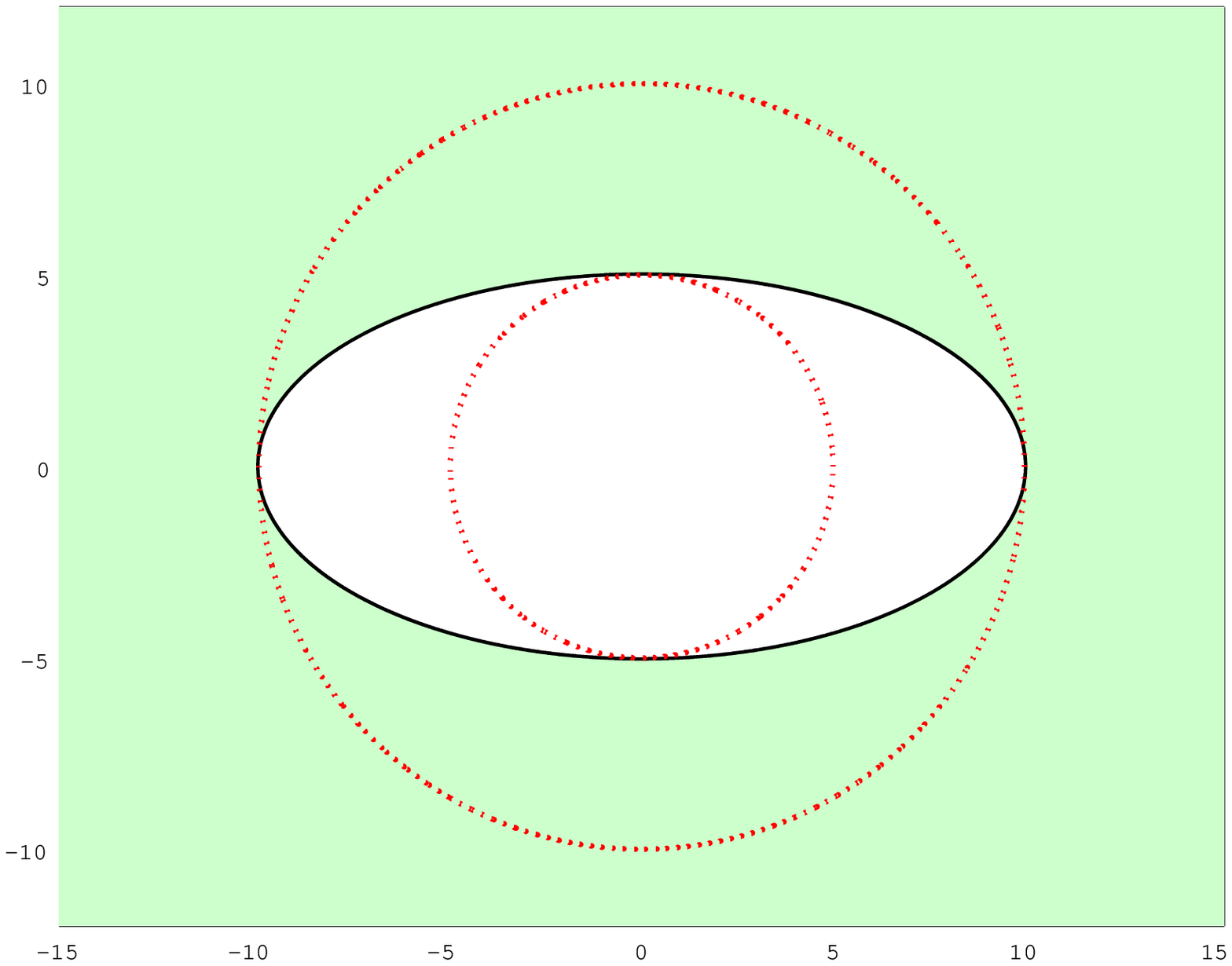}
\includegraphics[width=0.3\textwidth]{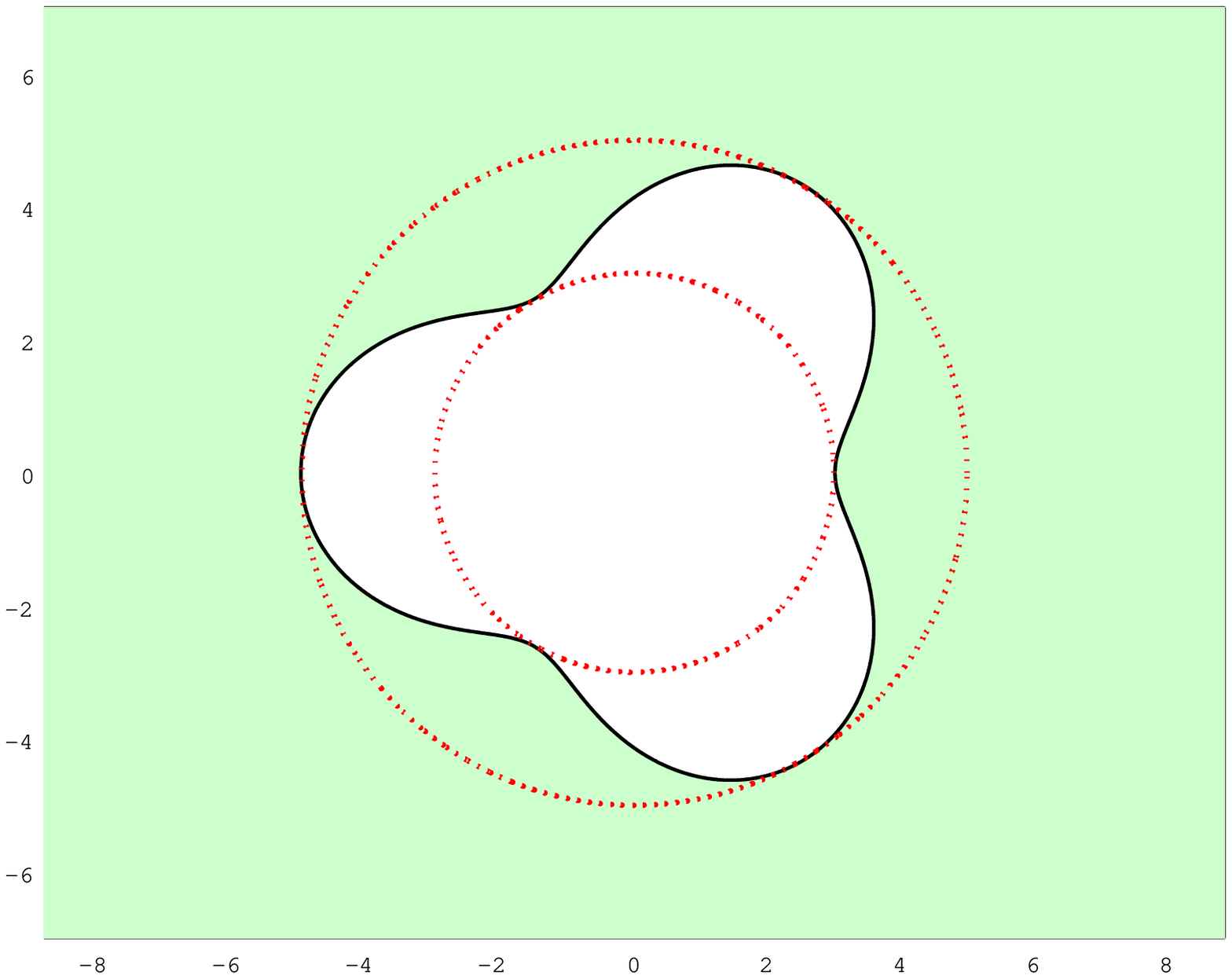}

\hfil (a) \hfil (b) \hfil (c) \hfil
\caption{Three domains}\label{fig:domain}
\end{figure}

We can confirm that the numerical results for $k=1$
which are shown below
are proper by comparing them with the corresponding 
theoretical results obtained in the previous section.

Figure~\ref{fig:condsr0} shows 
the condition number $\cond_2(S)$ against $0<R_0\leq 15$
for each domain $\Omega_k$ ($k=1,2,3$) with $N=21$ and $M=10$.
\begin{figure}[H]
\centering
\includegraphics[width=0.8\textwidth]{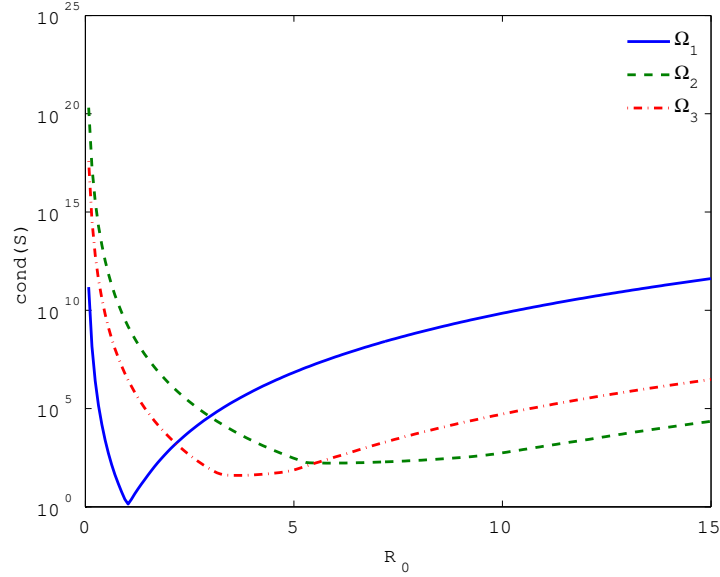}
\caption{Condition number of $S$ against $R_0$ 
for each domain $\Omega_k$ ($N=21,M=10$)}
\label{fig:condsr0}
\end{figure}
From this result, we can observe that
\[
 \min_{0<R_0\leq 15}\cond_2(S)=\begin{cases}
1.366 & \text{at } R_0=1.035 \quad (k=1)\\
167.439 & \text{at } R_0=5.850 \quad (k=2)\\
39.481 & \text{at } R_0=3.636 \quad (k=3)
\end{cases}.
\]
We can see that the minimal $\cond_2(S)$ and
the corresponding $R_0$ for $k=1$ coincide with
the result of Corollary~\ref{cor:optr}.
When we use the conventional Trefftz method ($R_0=1$),
the condition number $\cond_2(S)$ 
for each domain $\Omega_k$ ($k=1,2,3$) with $N=21$ and $M=10$
can be obtain as follows:
\[
 \cond_2(S)=\begin{cases}
1.4142 &  (k=1)\\
1.9430\times 10^9 &  (k=2)\\
3.6112\times 10^6 &  (k=3)
\end{cases}.
\]
Hence, we can see that the choice
of a suitable $R_0$ effects the drastic reduction
of $\cond_2(S)$.

\begin{figure}[H]
\centering
\includegraphics[width=0.8\textwidth]{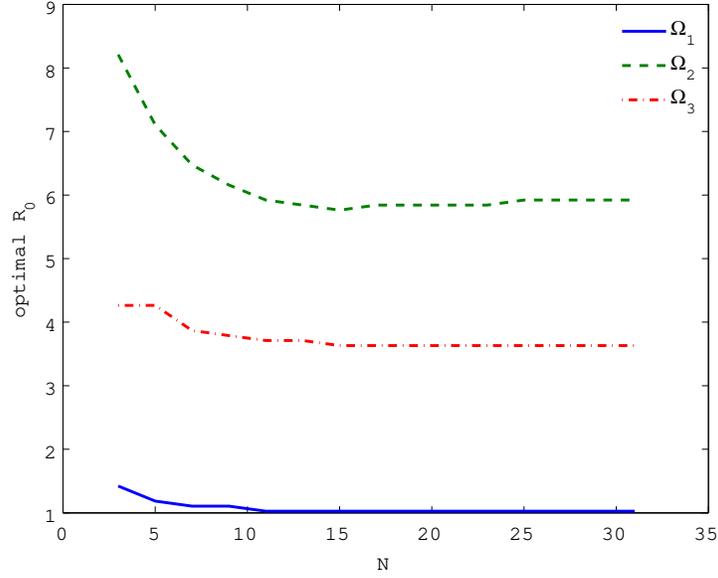}
\caption{Optimal $R_0$ against $N$ 
for each domain $\Omega_k$}
\label{fig:optr0n}
\end{figure}
Figure~\ref{fig:optr0n} shows $R_0^{\rm opt}$
against $N$ for each domain $\Omega_k$.
In this figure, $R_0^{\rm opt}$ seems to remain stable
as $N$ increases.
We can observe that
\[
 R_0^{\rm opt}\geq\rho_{\rm min}=\min_{0\leq\theta<2\pi}\rho(\theta),
\]
which is theoretically true when $k=1$ since
$R_0^{\rm opt}=2^{1/(2M)}\rho=2^{1/(N-1)}\rho$
 is given by Corollary \ref{cor:optr},
and $R_0^{\rm opt}\approx \rho=1$ for sufficiently large $N$.
From this result, we can conclude
that a suitable $R_0$ should be a little bit
greater than $\rho_{\rm min}$.

\begin{figure}[H]
\centering
\includegraphics[width=0.8\textwidth]{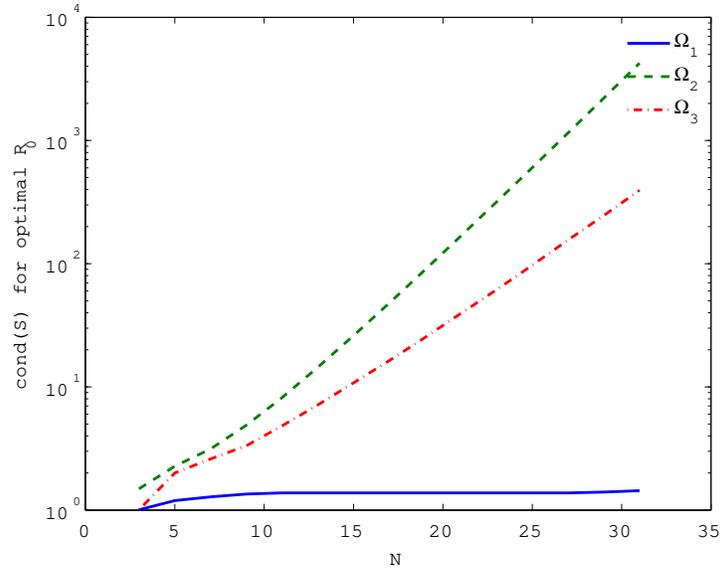}
\caption{$\cond_2(S)$ for $R_0^{\rm opt}$ against $N$ 
for each domain $\Omega_k$}
\label{fig:condsn}
\end{figure}
Figure~\ref{fig:condsn} shows
$\cond_2(S)$ for $R_0^{\rm opt}$ against $N$
for each domain $\Omega_k$.
Figure~\ref{fig:condsnr01} shows $\cond_2(S)$ against $N$
for each domain $\Omega_k$ when $R_0=1$.
Comparing Figure~\ref{fig:condsn} with Figure~\ref{fig:condsnr01},
we can confirm that the suitable characteristic length $R_0$
can drastically reduce the condition number.
\begin{figure}[H]
\centering
\includegraphics[width=0.8\textwidth]{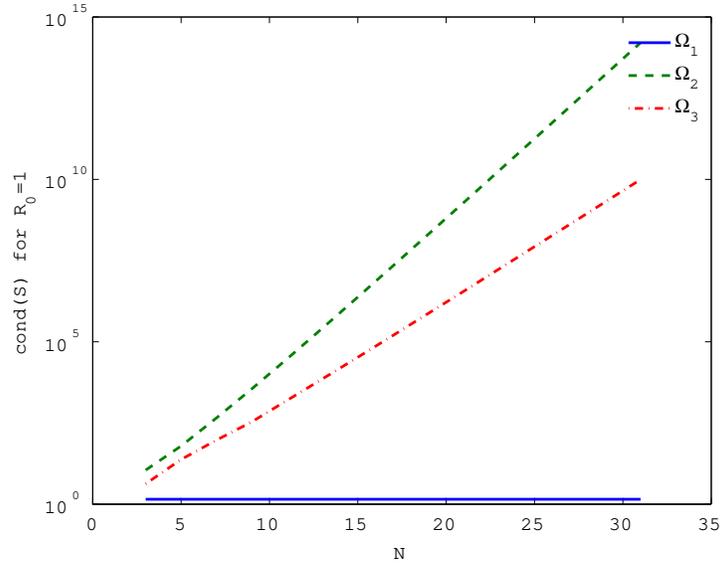}
\caption{$\cond_2(S)$ against $N$ 
for each domain $\Omega_k$ ($R_0=1$)}
\label{fig:condsnr01}
\end{figure}

\subsection{Condition number of the transform matrix from the MMFS to the MTM}
\begin{figure}[H]
\centering
\includegraphics[width=0.8\textwidth]{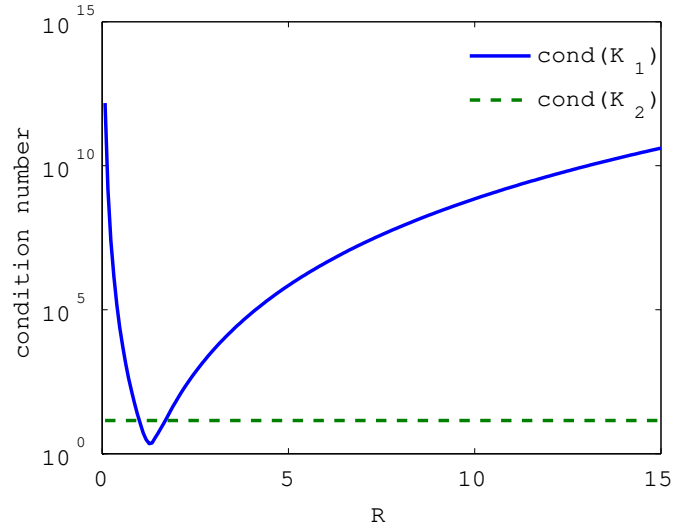}
\caption{$\cond_2(K_1)$ v.s. $\cond_2(K_2)$ against $R$ 
($N=21$)}
\label{fig:condkr}
\end{figure}
\begin{figure}[H]
\centering
\includegraphics[width=0.8\textwidth]{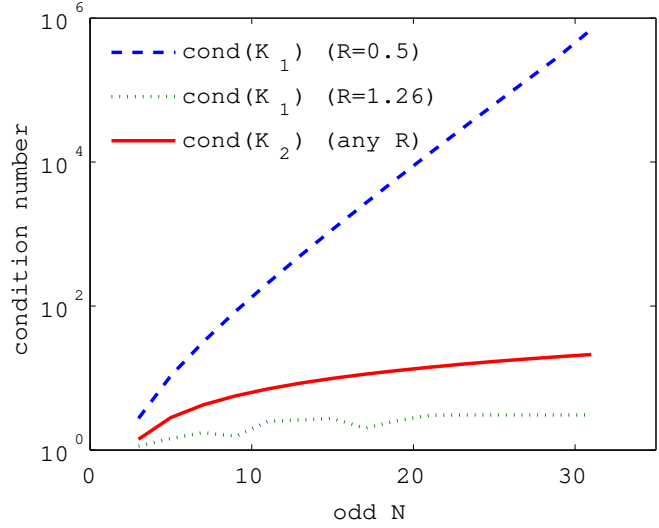}
\caption{$\cond_2(K_1)$ ($R=0.5,1.26$) v.s. $\cond_2(K_2)$ against $N$}
\label{fig:condkn}
\end{figure}
The matrix $K$ depends only on $(N,M,R,R_0)$.
We notice that $K$ is independent of $\Omega$.
Figure~\ref{fig:condkr} shows both of $\cond_2(K_1)$ and
 $\cond_2(K_2)$
against $R$ when $N=21$.
From this figure, we can observe that $\cond_2(K_1)$
has the minimal value $2.22$ at $R=1.26$,
which is smaller than $\cond_2(K_2)=14.1421$ at any $R$.
But, we cannot always take $R$ that minimizes $\cond_2(K_1)$
since the source radius $R$ must be smaller than $\rho_{\rm min}$.
Figure~\ref{fig:condkn} shows 
both of $\cond_2(K_1)$ for $R=0.5,1.26$ and
 $\cond_2(K_2)$ against odd $N$.
We see from Figures \ref{fig:condkr} and \ref{fig:condkn}
that $\cond_2(K_2)$
is much smaller than $\cond_2(K_1)$
for most of $R$, from which
we confirm the efficiency of $K_2$.

\subsection{Comparison of the MTM, the MFS and the MMFS}
We assume that the exact solution is given by
\[
 u(x,y)=\exp\left(\frac x{x^2+y^2}\right)
\cos\left(\frac y{x^2+y^2}\right)
\]
in the exterior domain $\Omega=\Omega_3$ outside
the epitrochoid boundary defined in the beginning of the section.
We confirm the asymptotic behavior of $u$ at infinity:
\begin{align*}
u(x,y)&=\exp\left(\frac x{x^2+y^2}\right)
\cos\left(\frac y{x^2+y^2}\right)\\
&=\exp\left(\frac{\cos\theta}r\right)
\cos\left(\frac{\sin\theta}r\right) \\
&=\left[
1+\frac{\cos\theta}r+O\left(\frac 1{r^2}\right)
\right]
\left[
1-\frac 12\left(\frac{\sin\theta}r\right)^2+O\left(\frac 1{r^4}\right)
\right]\\
&=1+O(r^{-1}),\qquad r\to\infty,
\end{align*}
which does not satisfy 
the condition (\ref{eq:inf}).
Hence, we use the MFS to solve the exterior problem
with respect to $\tilde u:=u-1$, which satisfies (\ref{eq:inf}).
Then, we can obtain the approximate solution to $u=\tilde u+1$.

We apply the following five methods for solving the problem:
\begin{enumerate}
\item The MTM

Solve $S\bm y=\bm f$ to substitute $\bm y$ into (\ref{eq:mtm2_ex})

\item The conventional MFS with the conventional basis functions
(CMFS-CBF)

Solve
 $A\bm w=\bm f$ to substitute $\bm w$ into (\ref{eq:aprx21})
\item The conventional MFS with the modified basis functions
(CMFS-MBF)

Solve
 $\hat A\bm w=\bm f$ to substitute $\bm w$ into (\ref{eq:cmfs-mbf})

\item The modified MFS with the conventional basis functions
(MMFS-CBF)

Solve
 $SK_2\bm w=\bm f$ to substitute $\bm w$ into (\ref{eq:aprx21})

\item The modified MFS with the modified basis functions
(MMFS-MBF)

Solve
 $SK_2\bm w=\bm f$ to substitute $\bm w$ into (\ref{eq:cmfs-mbf})
\end{enumerate}

We take $(N,M)=(19,9)$.
First, we confirm the accuracy of the solution obtained by
all the methods.
The absolute error between the numerical and the exact solutions
on the circle whose radius is $r$
is defined as
\[
 e(r,\theta)=
|\tilde u(r\cos\theta,r\sin\theta)
-u(r\cos\theta,r\sin\theta)|,
\]
where $\tilde u$ stands for the numerical solution
by one of the five methods.
Figure~\ref{fig:r10err} shows the absolute errors $e(10,\theta)$
for $R=0.5,1$.
We can see that the MMFS-MBF is the best accurate
among all the methods.
\begin{figure}[H]
\includegraphics[width=0.45\textwidth]{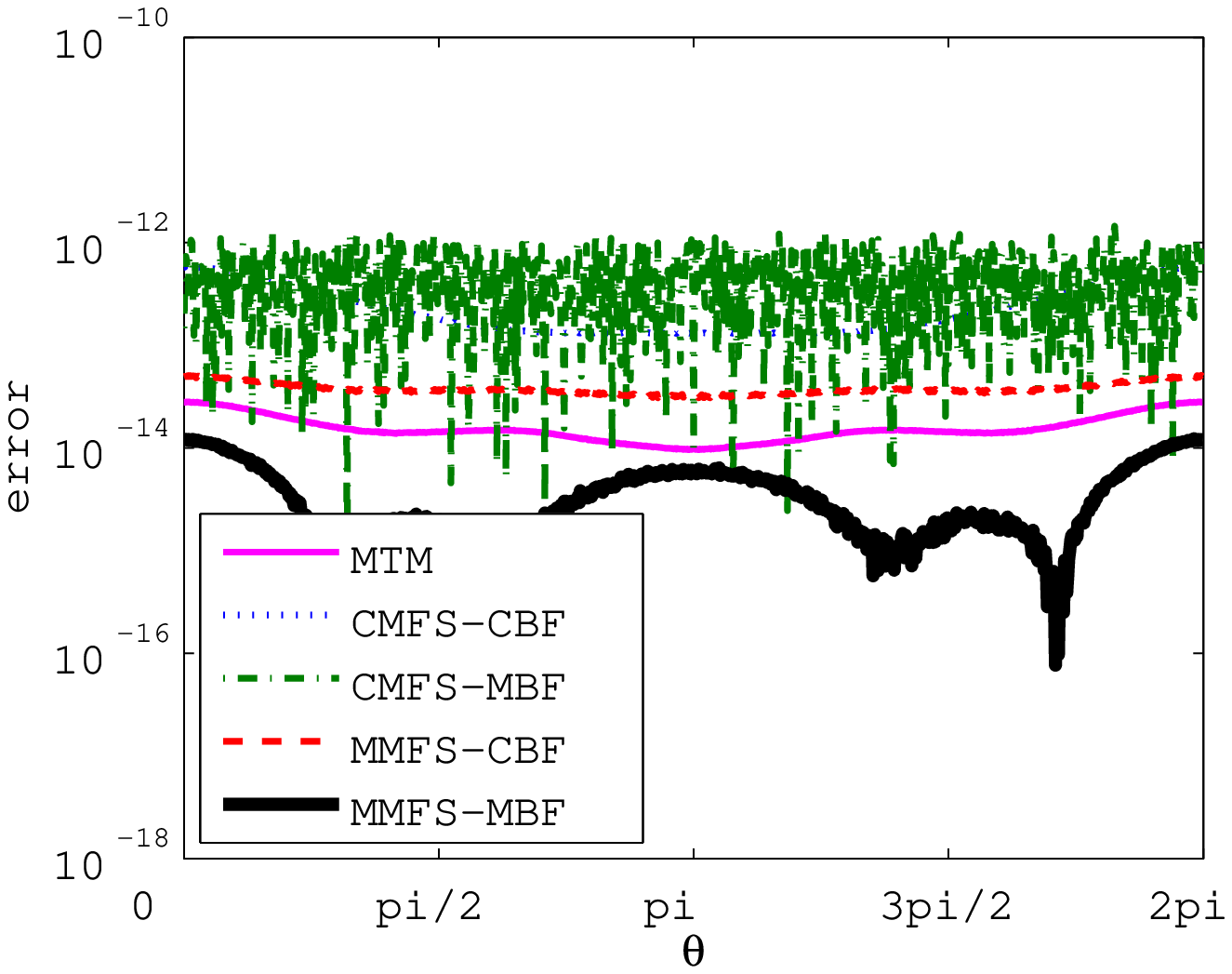}
\includegraphics[width=0.45\textwidth]{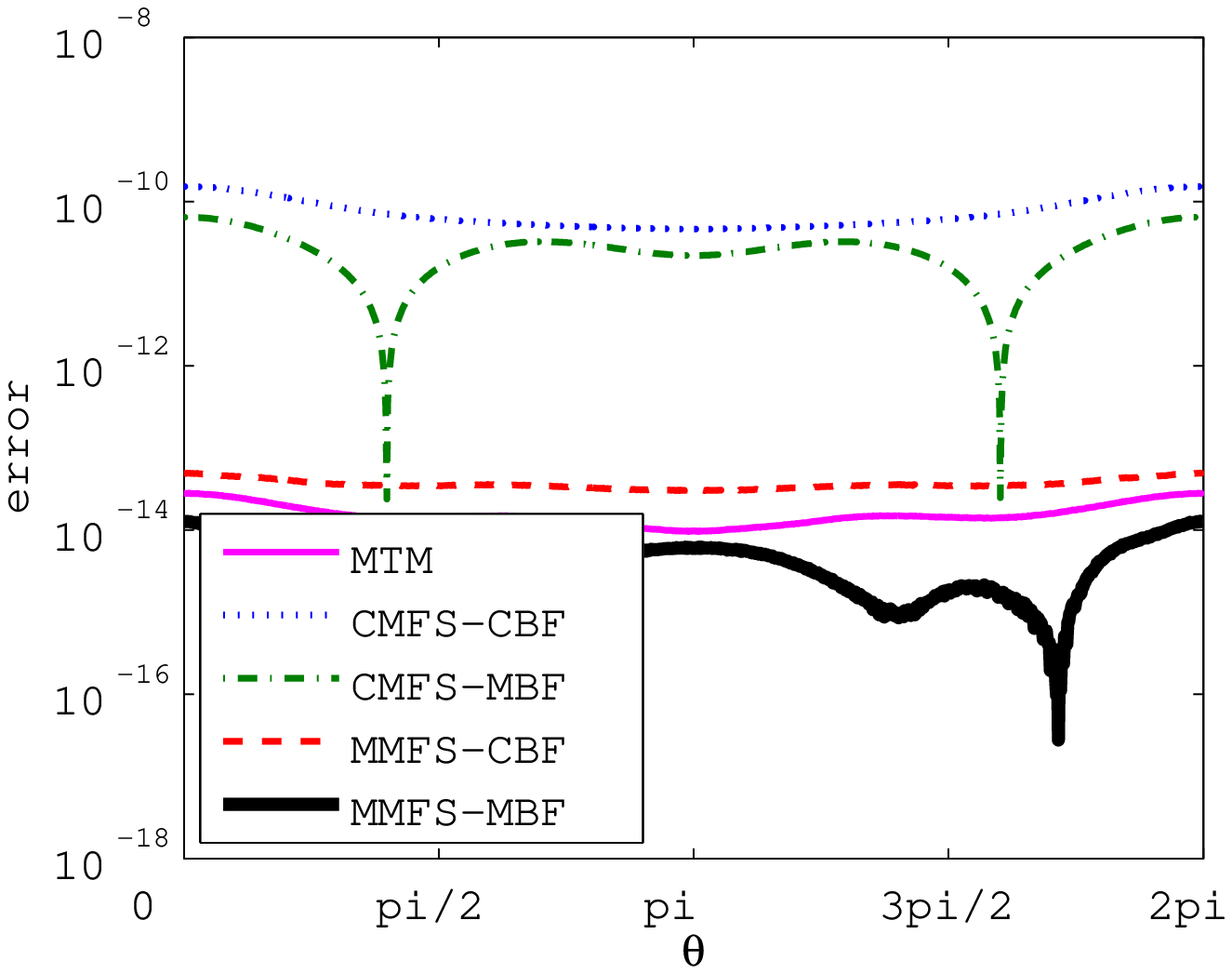}
\caption{Absolute error ($(R,N)=(1,19)$, $r=0.5$ (left), $r=1$ (right))}
\label{fig:r10err}
\end{figure}

Next, we confirm the accuracy of the solution obtained by
all the methods in the whole exterior domain.
We define the following maximum error: 
\[
 e(r):=\max_{0\leq\theta<2\pi}e(r,\theta).
\]

\begin{figure}[H]
\begin{center}
\includegraphics[width=0.45\textwidth]{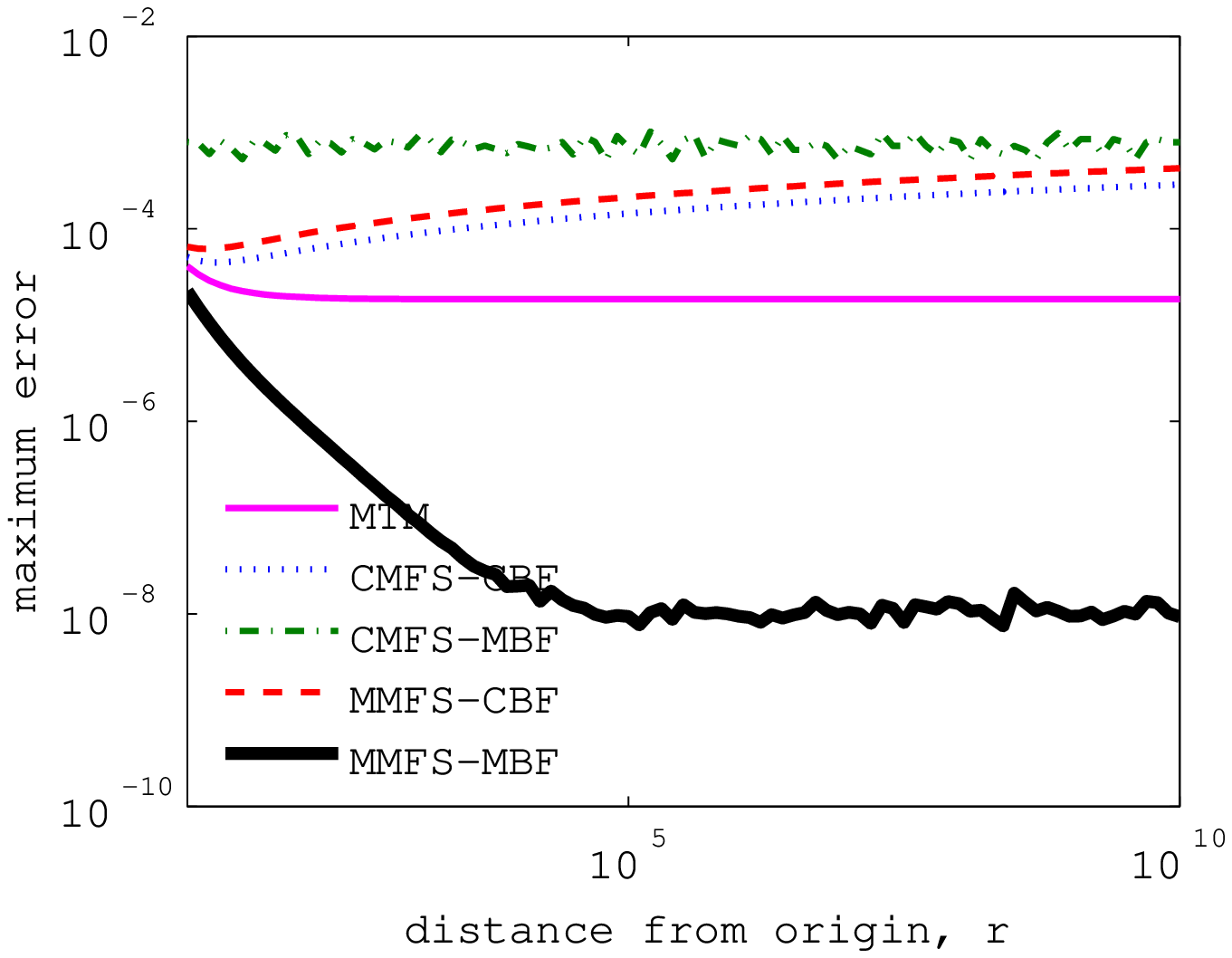}
\includegraphics[width=0.45\textwidth]{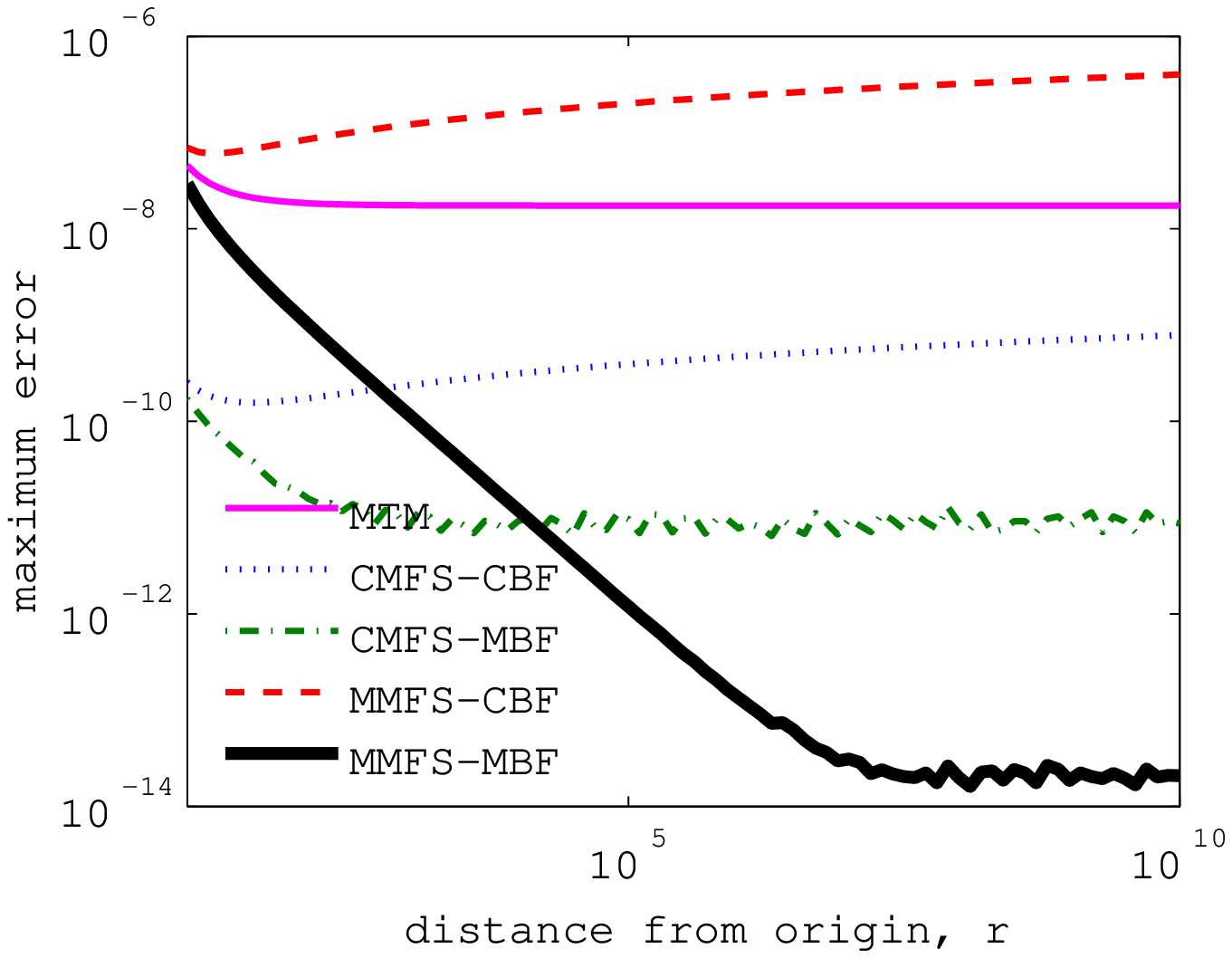}

(a) $(R,N)=(0.002,7)$ \qquad\qquad (b) $(R,N)=(0.2,11)$

\includegraphics[width=0.45\textwidth]{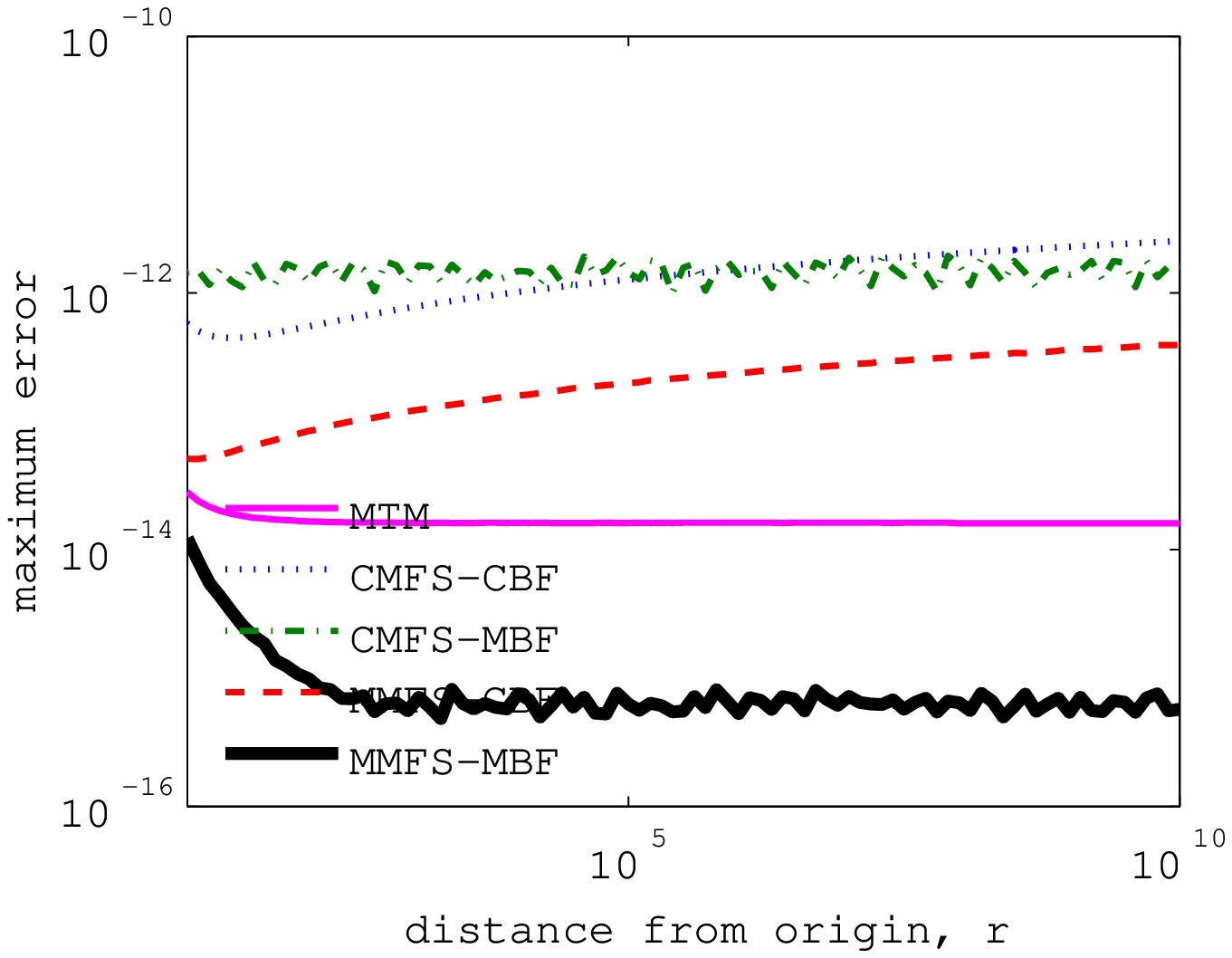}
\includegraphics[width=0.45\textwidth]{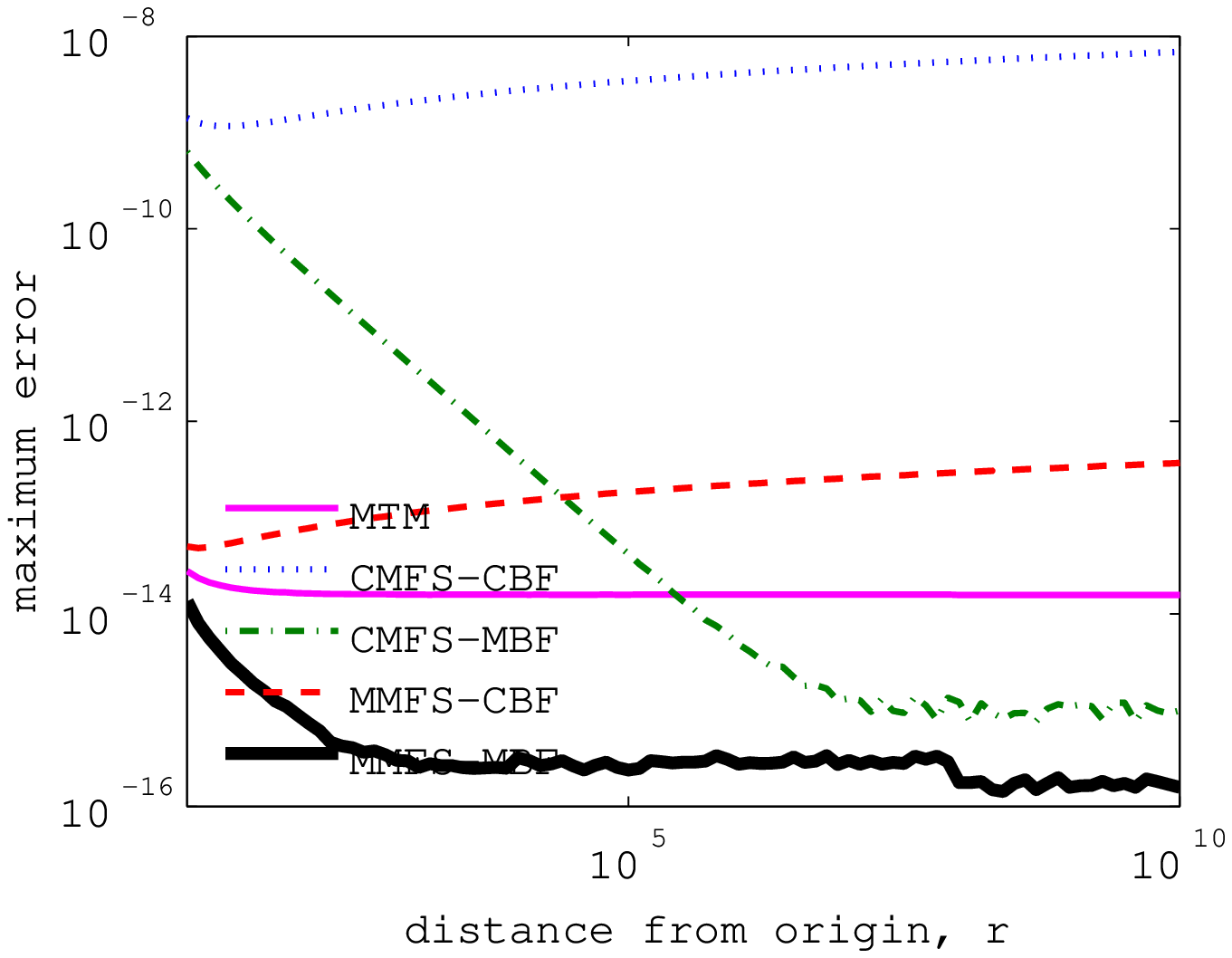}

(c) $(R,N)=(0.5,19)$ \qquad\qquad (d) $(R,N)=(1.2,19)$

\includegraphics[width=0.45\textwidth]{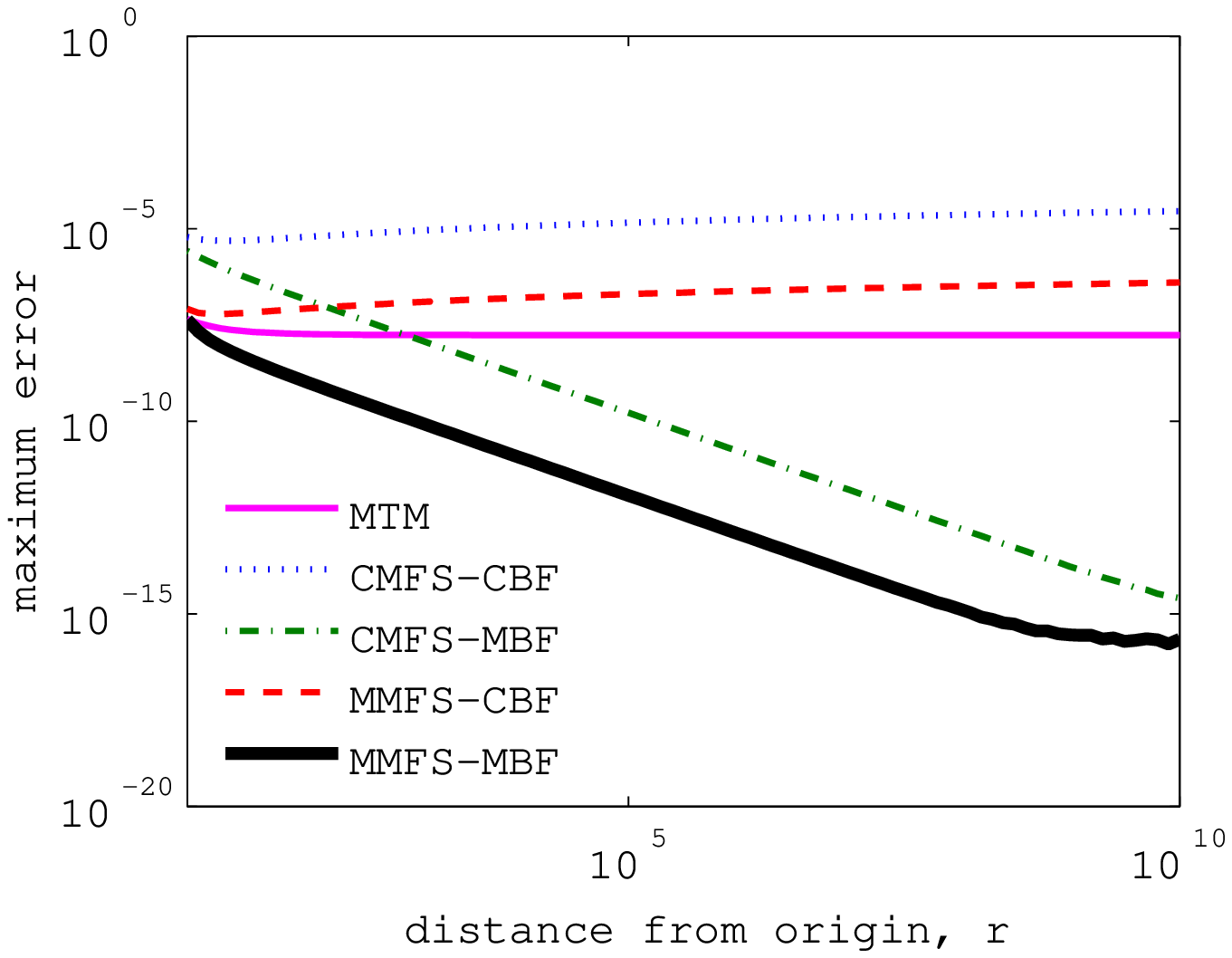}
\includegraphics[width=0.45\textwidth]{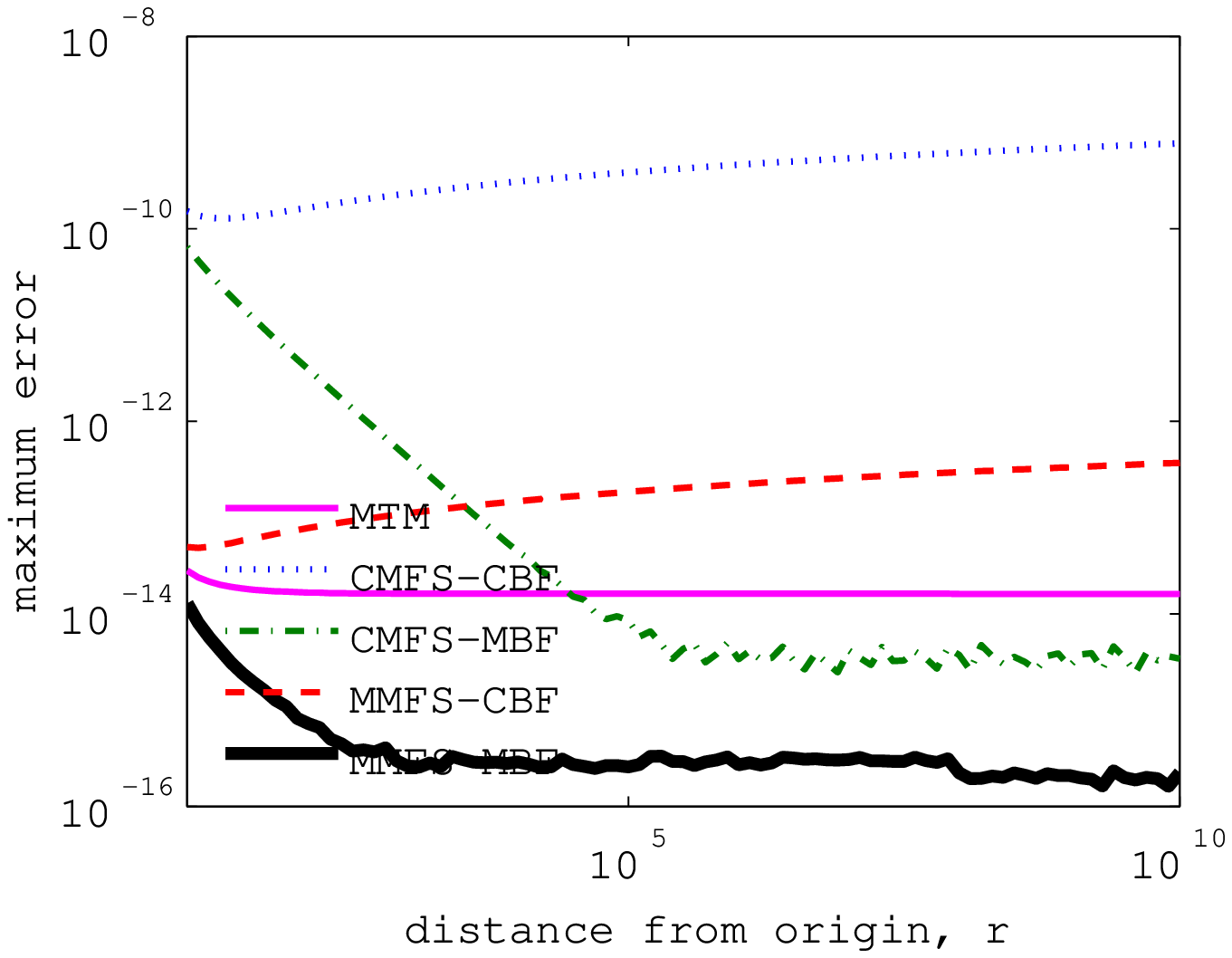}

(e) $(R,N)=(1.2,11)$ \qquad\qquad (f) $(R,N)=(1.0,19)$
\end{center}
\caption{Maximum errors $e(r)$ against the distance $r$ from the origin}
\label{fig:err-r}
\end{figure}
Figure~\ref{fig:err-r} shows the errors $e(r)$
against $r\in[10,10^{10}]$
with
(a) $(R,N)=(0.002,7)$,
(b) $(R,N)=(0.2,11)$,
(c) $(R,N)=(0.5,19)$,
(d) $(R,N)=(1.2,19)$,
(e) $(R,N)=(1.2,11)$,
(f) $(R,N)=(1.0,19)$.
We can observe from the figure that
the both of the errors for the CMFS-CBF and the MMFS-CBF
increase as $r$ increases,
while
the both of the errors for the CMFS-MBF and the MMFS-MBF
decrease as $r$ increases,
Hence, we know that the modified basis functions
are suitable for approximating the solution in the whole
exterior domain.

Moreover, we can see from the same figure that
in both of the cases where the conventional and the modified
basis functions are used,
the accuracy of the MMFS is improved more
than the accuracy of the conventional MFS.

For small $r$, the MMFS and the modified basis function
do not always give better accuracy.
But, they give much better accuracy by taking suitable
parameters $R$ and $N$.
The accuracy of the MMFS-MBF is
the best 
for the whole $r$ in almost cases.

Therefore, we can conclude that the MMFS-MBF is the best
method among the five methods
if we need to obtain a highly accurate solution
in the whole domain.

\subsection{Condition numbers corresponding to the MFS and the MMFS}
\begin{figure}[H]
\includegraphics[width=0.8\textwidth]{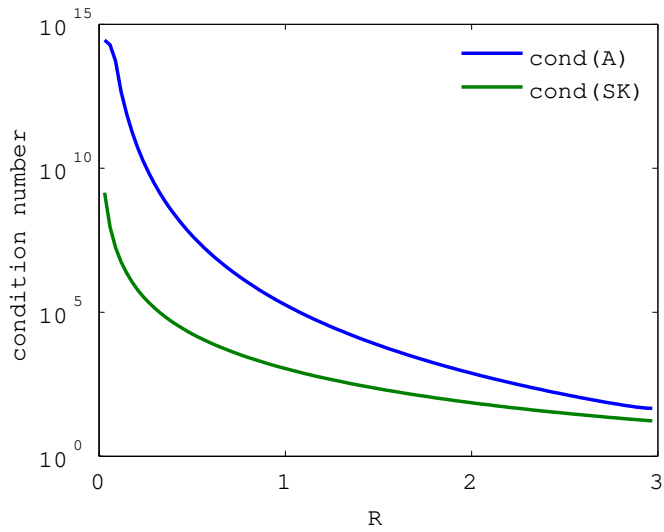}
\caption{$\cond_2(A)$ v.s. $\cond_2(SK)$ against $R$ ($N=9$)}
\label{fig:condaskr}
\end{figure}
\begin{figure}[H]
\includegraphics[width=0.8\textwidth]{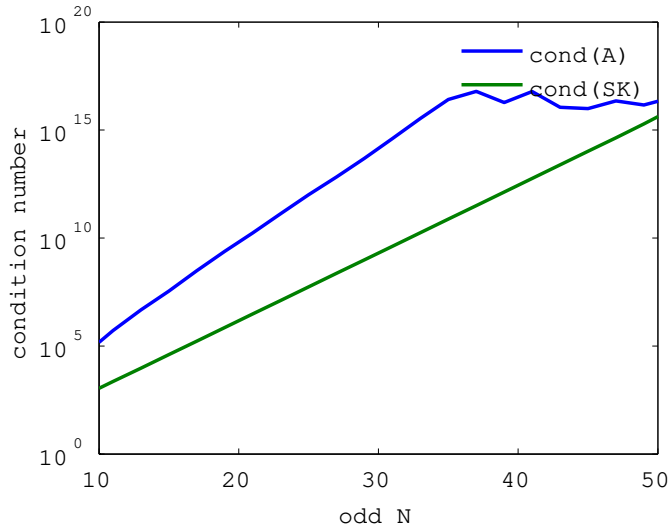}
\caption{$\cond_2(A)$ v.s. $\cond_2(SK)$ against $N$ ($R=1.2$)}
\label{fig:condaskn}
\end{figure}
Figure \ref{fig:condaskr} compares the condition numbers
of $A$ and $SK$ corresponding to the CMFS and the MMFS, respectively,
when $N=9$.
Figure \ref{fig:condaskn} compares the condition numbers
of $A$ and $SK$ for $R=1.2$, from which we can observe
$\cond_2(A)=8.552\times 2.752^N$
and $\cond_2(SK)=1.248\times 2.058^N$.
Therefore, $\cond_2(SK)$ is smaller than $\cond_2(A)$.
It is concluded that we can successfully reduce
the condition number by using the MMFS.

\section{Conclusions}
We have proposed the MMFS with the modified basis functions
for solving
the exterior boundary value problem, based on 
the MTM.
Under the assumption of the circular boundary,
the condition number corresponding to the MTM is 
mathematically shown.
Then, the optimal characteristic length $R_0$ that minimizes
the condition number is given in the mathematical form.
The uniqueness of the approximate solutions by the MTM
and the MMFS is also proven.

The numerical experiments shows that
the MMFS-MBF proposed in the paper
is a more accurate method
than the MTM, the CMFS-CBF and the CMFS-MBF.
The condition number corresponding to the MMFS
is smaller than the one corresponding to the CMFS.
It is concluded that the MMFS-MBF is an efficient method.

As future works, we will 
give some mathematical expressions of
condition numbers
that are not given in this paper.
More methodical approach than one in this paper
will be appeared in another paper.
It is important to prove the reason why the MMFS
can improve the accuracy.

\section*{Acknowledgements}
The authors gratefully acknowledge the financial support
of the National Science Council of Taiwan through
the grant No.\ NSC98--2811--E--002--092.

\end{document}